\documentclass[reqno]{amsart}
\usepackage{amssymb}
\usepackage{amsmath}
\makeatletter
\@addtoreset{equation}{section}
\makeatother

\renewcommand\thefigure{\thesection.\@arabic\c@figure}
\renewcommand\thetable{\thesection.\@arabic\c@table}

\newtheorem{theorem}{Theorem}[section]
\newtheorem{lemma}[theorem]{Lemma}
\newtheorem{proposition}[theorem]{Proposition}
\newtheorem{corollary}[theorem]{Corollary}

\newtheorem{definition}[theorem]{Definition}

\newcommand{\mc}[1]{{\mathcal #1}}
\newcommand{\mf}[1]{{\mathfrak #1}}
\newcommand{\mb}[1]{{\mathbf #1}}
\newcommand{\bb}[1]{{\mathbb #1}}
\newcommand{\bs}[1]{{\boldsymbol #1}}

\newcommand{\<}{\langle}
\renewcommand{\>}{\rangle}

\renewcommand{\Cap}{{\rm cap}}

\begin{document}

\title[]{Meta-stability and condensed zero-range processes on finite
  sets}

\author{J. Beltr\'an, C. Landim}

\address{\noindent IMCA, Calle los Bi\'ologos 245, Urb. San C\'esar
  Primera Etapa, Lima 12, Per\'u and PUCP, Av. Universitaria cdra. 18,
  San Miguel, Ap. 1761, Lima 100, Per\'u. 
\newline e-mail: \rm
  \texttt{johel@impa.br} }

\address{\noindent IMPA, Estrada Dona Castorina 110, CEP 22460 Rio de
  Janeiro, Brasil and CNRS UMR 6085, Universit\'e de Rouen, Avenue de
  l'Universit\'e, BP.12, Technop\^ole du Madril\-let, F76801
  Saint-\'Etienne-du-Rouvray, France.  \newline e-mail: \rm
  \texttt{landim@impa.br} }


\begin{abstract}
  We propose a definition o meta-stability and obtain sufficient
  conditions for a sequence of Markov processes on finite state spaces
  to be meta-stable. In the reversible case, these conditions reduce
  to estimates of the capacity and the measure of certain meta-stable
  sets. We prove that a class of condensed zero-range processes with
  asymptotically decreasing jump rates is meta-stable.
\end{abstract}

\maketitle

\section{Introduction}
\label{sec0}

The zero-range dynamics is a Markov process which models the evolution
of indistinguishable particles on some set $S$. It can be informally
described as follows. Set $S= \Lambda_L$, the one dimensional discrete
torus with $L$ points, fix a finite range probability measure $p$ on
$\bb Z$ and a rate function $g: \bb N\to \bb R_+$ such that $g(0)=0$,
$g(k)>0$ for $k\ge 1$. If there are $n$ particles at a site $x$ of
$\Lambda_L$, at rate $p(y) g(n)$ one of them jumps from $x$ to $x+y$.
The name ``zero-range'' derives from the observation that,
infinitesimally, the interaction is only with respect to those
particles at the particular vertex.  The case when $g(k)$ is
proportional to $k$ describes the situation of completely independent
particles.

Denote by $\eta$ the configurations of the state space $\bb
N^{\Lambda_L}$ and by $\eta(x)$, $x\in\Lambda_L$, the number of
particles at site $x$ for the configuration $\eta$. To describe the
invariant states of the process, let $Z: [0, \varphi^*) \to \bb R_+$
be the partition function given by
\begin{equation*}
Z(\varphi) \;=\; \sum_{k\ge 0} \frac {\varphi^k}{g(k)!}\;,
\end{equation*}
where $g(k)! = g(1) \cdots g(k)$, and where $\varphi^*$ stands for the
radius of convergence of $Z$. An elementary computation shows that the
product measures $\{\nu_\varphi : 0\le \varphi <\varphi^*\}$ defined
by
\begin{equation*}
\nu_{\varphi} (\eta) \;=\; \prod_{x\in \Lambda_L} 
\frac 1{Z(\varphi)} \frac{\varphi^{\eta(x)}}{g(\eta(x))!}
\end{equation*}
are stationary states for the Markov dynamics described above. The
parameter $\varphi$ is called the fugacity.

The density of particles under the stationary state $\nu_\varphi$,
denoted by $\rho(\varphi)$, is given by
\begin{equation*}
  \rho(\varphi) \;=\; E_{\nu_{\varphi}} [ \eta(x)]\;=\;
\frac 1{Z(\varphi)} \sum_{k\ge 0} k \, \frac {\varphi^k}{g(k)!}
\;=\; \varphi\, \partial_\varphi \log Z(\varphi)\;.
\end{equation*}
The density $\rho: [0, \varphi^*) \to \bb R_+$ is a strictly
increasing function of the fugacity.  Denote $\varphi(\rho)$ the
inverse function and by $\rho^* = \lim_{\varphi \to\varphi^*}
\rho(\varphi)$ the critical density.  By \cite[Lemma II.3.3]{kl},
$\rho^* = \infty$ whenever $\varphi^* = \infty$. In contrast, if
$\varphi^* < \infty$ both cases can occur.

Assume that $\rho^*<\infty$, i.e., that we have a phase transition.
By \cite[Lemma II.3.3]{kl}, $Z(\varphi^*) < \infty$ and the measure
$\nu_{\varphi^*}$ is well defined. Therefore, there is a product
stationary state if and only if the density is less than or equal to
the critical density $\rho^*$.

To examine the stationary states above the critical density, denote by
$\Sigma_{\Lambda_L,N}$, $N\ge 1$, the set of configurations with $N$
particles: $\Sigma_{\Lambda_L,N} = \{ \eta\in \bb N^{\Lambda_L} :
\sum_{x\in \Lambda_L} \eta (x)$ $ = N\}$.  Since the dynamics preserve
the total number of particles, if the random walk on $\Lambda_L$
associated to the transition probability $p$ is irreducible, the sets
$\Sigma_{\Lambda_L,N}$ are the ergodic components of the state space.
Let $\mu_{L,N}$ denote the unique stationary probability measure on
$\Sigma_{\Lambda_L,N}$.

For $\rho \le \rho^*$, as the total number of site increases to
infinity and the density is kept equal to $\rho$, the equivalence of
ensembles \cite[Appendix 2]{kl} states that the marginals of the
measure $\mu_{L,N}$ converges to the measure $\nu_{\varphi(\rho)}$.
More precisely, for any $n\ge 1$ and any $(k_1, \dots, k_n)$ in $\bb
N^n$,
\begin{equation*}
\lim_{L\to\infty} \mu_{L, [\rho L]} \big\{\eta : \eta(x_i) = k_i\,,
1\le i\le n \big\} \;=\; \nu_{\varphi(\rho)} 
\big\{\eta : \eta(x_i) = k_i\,, 1\le i\le n \big\} \;.
\end{equation*}
In this formula, $[a]$ stands for the integer part of $a\in\bb R$.

Above the critical density, in the stationary state, all additional
particles concentrate on \emph {one single} site \cite{evans, jmp,
  gss, fls}. In the remaining sites, particles are distributed
according to the grand canonical measure at the critical density
\cite{gss, al}: for $\rho > \rho^*$,
\begin{eqnarray*}
\!\!\!\!\!\!\!\!\!\!\!\!\! &&
\lim_{L\to\infty} \mu_{L, [\rho L]} \big\{\eta : \eta(x_i) = k_i\,,
1\le i\le n \big\} \;=\; \nu_{\varphi(\rho^*)} 
\big\{\eta : \eta(x_i) = k_i\,, 1\le i\le n \big\} \;, \\
\!\!\!\!\!\!\!\!\!\!\!\!\! && \quad 
\lim_{L\to\infty} \mu_{L, [\rho L]} \big\{\max_{1\le x\le L}
\eta(x) \ge [(\rho-\rho^*)L] \big\} \;=\; 1 \;.
\end{eqnarray*}

There is a huge physical literature on condensation. In this respect,
the zero-range process and similar interacting particle systems have
been used to model traffic flow, phase separation dynamics and
sandpile models (\cite{emz} and references therein).

In this article we examine the time evolution of the site which
concentrates all additional particles in the case where the total
number of sites is kept fixed, particles are allowed to jump to any
site with uniform probability and the total number of particles
increases to infinity.  We prove that in a proper time scale the
position of the site evolves according to a random walk on
$\Lambda_L$.

To prove this result we first propose a simple analytical definition
of meta-stability to be compared with the pathwise approach introduced
in \cite{cgov} and the spectral definition in \cite{b2}.  A sequence
of Markov processes $\{\eta^N_t : t\ge 0\}$ indexed by a parameter $N$
is said to be meta-stable if the state space $E_N$ can be partitioned
in a fixed finite number of sets $\Delta_N, \mc E^1_N, \dots , \mc
E^\kappa_N$ with the following two properties.

\renewcommand{\theenumi}{\alph{enumi}}
\renewcommand{\labelenumi}{(\theenumi)}

\begin{enumerate}
\item The time spent in $\Delta_N$ in any interval of time $[0,t]$
  vanishes as $N\uparrow\infty$.
  
\item Denote by $\xi^N_t$ the trace of the process $\eta^N_t$ on
  $E_N\setminus \Delta_N$. For $x$ in $E_N\setminus \Delta_N$, let
  $\Psi_N (x) = j$ if $x$ belongs to $\mc E^j_N$, $1\le j\le \kappa$.
  The (not necessarily Markovian) process $\Psi_N(\xi^N_t)$ converges
  to a Markovian process as $N\uparrow\infty$.
\end{enumerate}

If there is a well deeper than all others, as it is the case in the
Curie-Weiss model of \cite{cgov} for some parameters, the asymptotic
Markov process has one absorbing state. In the condensed zero-range
processes of Subsection \ref{2.4}, in contrast, there are several well
at the same maximal depth and no absorbing state appears in the limit.

We obtain sufficient conditions for a sequence of Markov processes on
finite state spaces to be meta-stable. In the reversible case, these
conditions reduce to estimates on the capacity and on the measure of
the meta-stable sets $\mc E^j_N$. We then prove that zero range
processes associated to a class of decreasing rate functions
$g(\cdot)$ are meta-stable.

\section{Notation and Results}
\label{sec1}

\subsection{Meta-stability}
Let $(E_N)_{N\ge 1}$ be a sequence of finite state spaces. For each
fixed $N\ge 1$ consider an irreducible $E_N$-valued Markov process
$(\bb P_\eta^N)_{\eta\in E_N}$ on the path space $D(\bb R_+,E_N)$ with
generator
$$
L_N(f)(\eta) \;=\; \sum_{\xi\in E_N}\big\{f(\xi)-f(\eta)\big\}\,
R_N(\eta,\xi)\;.
$$
Let $\bb E_{\,\eta}^{\,N}$ denote the expectation with respect to $\bb
P^N_{\eta}.$ For $t\ge 0$, let $\eta^N_t$ stand for the projection
$\eta^N_t:D(\bb R_+,E_N)\mapsto E_N$. As usual, we shall also use
$\eta^N_t$ to refer to the Markov Process $(\bb P_\eta^N)_{\eta\in
  E_N}$.

Fix an integer $\kappa\ge 1$ such that
$$
\kappa \;<\; \liminf_{N\to\infty} | E_N |\;,
$$
where $|E_N|$ stands for the cardinality of $E_N$. Let
$S=\{1,2,\dots,\kappa\}$. For each $E_N$ and each $x\in S$ fix a
subset $\mc E^x_N\subseteq E_N$ such that $\mc E^x_N\cap \mc
E^y_N=\emptyset$ for all $x\neq y$. Let us denote $\mc E_N=\cup_{x\in
  S}\mc E^x_N$ and $\Delta_N=E_N\setminus \mc E_N$, so that we get the
partition 
$$
E_N\;=\; \big(\bigcup_{x\in S} \mc E^x_N\big)
\cup\Delta_N\;.
$$ 
Let $\Psi_N:\mc E_N\mapsto S$ be given by
$$
\Psi_N(\eta) \;=\; \sum_{x\in S} x\, \mathbf 1\{\eta \in \mc E^x_N\}\;.
$$

In order to introduce the notion of meta-stability we need to recall
the definition of the trace of a Markov process on a subset of the
state space. Let $F$ be a subset of $E_N$. For each $t\ge 0$, let $\mc
T^{F}_t:D(\bb R_+,E_N)\mapsto\bb R_+$ be the time the process stayed
in the set $F$ in the interval $[0,t]$:
$$
\mc T^{F}_t(\omega)\;:=\;\int_{0}^t \mathbf{1}\{\omega(s)\in F\}ds
$$
and let $\mc S^{F}_t$ be the generalized inverse of $\mc T^{F}_t:$
$$
\mc S^{F}_t(\omega)\;:=\;\sup\{s\ge 0 : \mc T^{F}_s(\omega)\le t\}.
$$
For any path $\omega\in D(\bb R_+,E_N)$ such that $\omega(0)\in F$, we
define a new path $\mc R^{F}(\omega)\in D(\bb R_+,F)$ by $\mc
R^{F}(\omega)(t)=\omega(S_t(\omega))\,$. It is well known that, if for
each $\eta\in F$ we denote
$$
\bb P^{F}_{\eta} \;:=\; \bb P^{N}_{\eta}\circ(\mc R^{F})^{-1},
$$
then $(\bb P^{F}_{\eta})_{\eta\in F}$ is a strong Markov process with
path space $D(\bb R_+,F)$ (see \cite[V.2.11]{bg}). This Markov process
is called the trace of the Markov process $\eta^N_t$ on $F$. Let
$\eta^{F}_t$ denote the projection $\eta^{F}_t: D(\bb R_+,F)\mapsto
F$, for $t\ge 0\,$. $\eta^{F}_t$ will also be used to refer to the
trace of $\eta^N_t$ on $F$.

Consider the trace of $\eta^N_t$ on $\mc E_N$, referred to as
$\eta^{\mc E_N}_t$. For each $\eta\in \mc E_N$ denote by $\bb
Q^N_{\,\eta}$ the measure on $D(\bb R_+, S)$ induced by the process
$X^N_t:=\Psi_N(\eta^{\mc E_N}_t)$ when $\eta^{\mc E_N}_t$ starts at
$\eta$.

\renewcommand{\theenumi}{\Alph{enumi}}
\renewcommand{\labelenumi}{(\theenumi)}

\begin{definition}[\textbf{Meta-stability}]\label{metadef} The
  sequence of Markov processes $(\eta^N_t)_{N\ge 1}$ is said to be
  meta-stable with meta-states given by $(\mc E^x_N)_{x\in S}$ and
  limit given by a Markov process $(\bb Q_x)_{x\in S}$ on $D(\bb
  R_+,S)$ if
\begin{enumerate}
\item \label{metaa} For any $x\in S$ and any sequence $\eta_N\in\mc
  E^x_N$, $N\ge 1$, the sequence $\bb Q^N_{\,\eta_N}$ converges to $\bb
  Q_{\,x}.$

\item \label{metab} For any $t>0$, $$\lim_{N\to\infty} \max_{\eta\in
    \mc E_N}\bb E^N_{\eta}[\mc T^{\Delta_N}_t ]\;=\;0\;.$$
\end{enumerate}
\end{definition}

Condition (B) states that $\eta^N_t$ spends a vanishing period of time
on the subset $\Delta_N$. Condition (A) states that, after cutting off
the negligible time $\eta^N_t$ spends on $\Delta_N$, the asymptotic
evolution of the trace process is described by a Markov process on
$S$.

In fact, if the meta-stability behaviour of $\eta^N_t$ is assumed,
such cut off is no more necessary to observe the asymptotic
evolution. More precisely, consider the stochastic process $\hat
X^N_t$ with state space $S$ defined as
$$
\hat X^N_t\;=\; \Psi_N(\eta^N_{\sigma(t)}) \;,
$$
where $\sigma(t):=\sup\{s\le t : \eta^N_s\in \mc E_N\}.$ Note that
$\hat X^N_t$ is well defined whenever $\eta^N_t$ starts at a point in
$\mc E_N$. Denote by $\hat{\bb Q}^N_{\,\eta}$ the measure on $D(\bb
R_+,S)$ induced by $\hat X_t^N$ when $\eta^N_t$ starts at $\eta\in \mc
E_N.$

\begin{proposition}\label{corol}
  If the sequence of Markov processes $(\eta^N_t)_{N\ge 1}$ is
  meta-stable with meta-states given by $(\mc E^x_N)_{x\in S}$ and
  limit given by the Markov process $(\bb Q_x)_{x\in S}$, then for any
  $x\in S$ and any sequence $\eta_N\in\mc E^x_N$, $N\ge 1$, the
  sequence $\hat{\bb Q}^N_{\,\eta_N}$ converges to $\bb Q_{\,x}.$
\end{proposition}

We postpone the proof of this proposition to Section \ref{proof}.

\subsection{Entry times and meta-stability}
\label{sec2}

The purpose of this subsection is to provide sufficient conditions on
the sequence of Markov processes $\eta^N_t$ to ensure a nontrivial
meta-stability behavior of the dynamics.

To state such conditions we shall use entry times. For a subset $G$ of
$E_N$ we denote by $H_{G}:D(\bb R_+,E_N)\mapsto\bb R_+$ the entry time
in $G$ defined as
\begin{equation}\label{f10}
H_G(\omega) \;:=\; \inf\{t\ge 0 : \omega(t) \in G\},
\end{equation}
with the convention that $H_{\eta}=H_{\{\eta\}}$ when $G$ is a
singleton $\{\eta\}.$ By abuse of notation we also denote by
$H_{G}:D(\bb R_+,F_N)\mapsto[0,+\infty]$ the entry time in $G\subseteq
F_N$ for any $F_N$ subset of $E_N$.

Let $(R^{\mc E}_N(\eta,\xi) : \eta,\xi\in\mc E_N)$ be the transition
rates of $\eta^{\mc E_N}_t$ and let $L^{\mc E}_N$ be the corresponding
generator:
\begin{equation*}
L^{\mc E}_Nf(\eta) \;=\; \sum_{\xi\in\mc E_N} \{f(\xi) - f(\eta)\}
R^{\mc E}_N(\eta,\xi)\;.
\end{equation*}
For each pair $x,y\in S$, $x\not = y$ and $\eta\in\mc E^x_N$ we define
$$
R^{\mc E}_N(\eta,y) \;:= \; \sum_{\xi\in\mc E^y_N}R^{\mc E}_N(\eta,\xi)\;.
$$
Recall that $\eta^N_t$ is supposed to be irreducible. Denote by
$\nu_N$ its unique invariant probability measure. Let $r_N(x,y)$ be
the expectation of $R^{\mc E}_N(\cdot,y)$ with respect to $\nu_N$
conditioned on $\mc E^x_N$ :
\begin{equation*}\label{rnxy}
r_N(x,y) \; := \; \frac{1}{\nu_N(\mc E^x_N)}\sum_{\eta\in\mc E^x_N}
R^{\mc E}_N(\eta,y)\nu_N(\eta)\;.
\end{equation*}
We shall require that the following limit exists
\begin{equation*}\tag*{{\bf (C1)}}
r(x,y) \;:= \; \lim_{N\to\infty} r_N(x,y)\;.
\end{equation*}
for all $x,y\in S$. The limits $(r(x,y) : x,y\in S)$ become the
transition rates of the Markov process limit in condition (A).

The Markov property for the limit of $X^N_t$ follows from a suitable
ergodic assumption for the dynamics on each $\mc E^x_N$. For each set
$\mc E^x_N$, choose a site $\xi^x_N$ and let $M_N=\{\xi^x_N : x\in
S\}$ be the meta-stable sites. We shall assume that for each $x\in S$,
the time it takes for $\eta^N_t$ to visit the meta-stable site
$\xi^x_N$, once it enters $\mc E^x_N$ is of sufficiently small
order. More precisely, we shall assume that
\begin{equation*}\tag*{{\bf (C2)}}
\lim_{N\to\infty}\;\big( \max_{\eta\in\mc E^x_N}
R_N^{\mc  E}(\eta,y)\big) \; 
\big(\max_{\eta\in\mc E^x_N}\bb E^N_{\,\eta}[H_{\xi^x_N}]\big)\;=\;0\;,
\end{equation*}
for all $x, y \in S$. 

In this assumption we could replace $\bb E^N_{\,\eta}[H_{\xi^x_N}]$ by
the smaller quantity $\bb E^{\mc E}_{\,\eta}[H_{\xi^x_N}]$. The
difference, however, is the time spent in $\Delta_N$ which we require
to vanish in condition (B).

For condition (B), we need to regard, for each $x\in S$, the trace of
$\eta^N_t$ on $E_N\setminus\mc E^x_N$, which will be denoted by $(\bb
P^{N,x}_{\eta})_{\eta\in E_N\backslash\mc E^x_N}$. Denote by $\bb
E^{N,x}_{\,\eta}$ the expectation with respect to $\bb
P^{N,x}_{\eta}$. For each $x\in S$, let us denote by
$\partial_x\Delta_N$ the set of sites in $\Delta_N$ which may be
attained from a site in $\mc E^x_N$ after a jump.
$$
\partial_x\Delta_{N}\;:=\;\{\xi\in\Delta_N : 
R_N(\eta,\xi) > 0\; \textrm{for some}\; \eta \in \mc E^x_N \}\;.
$$
Let also $\check{\mc E}^x_N$ stands for $\mc E_N \setminus \mc
E^x_N$.  We shall require that
\begin{equation*}\tag*{{\bf (C3)}}
\lim_{N\to\infty}\;\max_{\eta\in\partial_x\Delta_{N}}
\bb E^{N,x}_{\,\eta}[H_{\check{\mc E}^x_N}] \;=\; 0\;
\end{equation*}
for all $x\in S$, where the maximum is equal to zero, if
$\partial_x\Delta_N$ is empty. In particular, {\bf (C3)} is
automatically satisfied if $\partial_x\Delta_N$ is empty for all $N\ge
0$.

We now state our first main result in this article.

\begin{theorem}\label{teo1}
  A sequence of Markov processes satisfying conditions {\rm {\bf
      (C1)}} -- {\rm {\bf (C3)}} is meta-stable in the sense of
  Definition \ref{metadef}, with meta-states given by $(\mc
  E^x_N)_{x\in S}$ and limit given by the Markov process $\{\bb
  Q_{\,x}\}_{x\in S}$ with transition rates $(r(x,y) : x,y\in S)$.
\end{theorem}

It is important to notice that in this theorem we do not exclude the
possibility that $\nu_N(\mc E^x_N)$ vanishes as $N\uparrow\infty$ for
some $x\in S$.  The proof of this theorem is given in Section
\ref{teo1}.

\subsection{The reversible case, potential theory}\label{cap}

In this subsection, we restrict ourselves to the case where $\nu_N$ is
a reversible probability measure for all $N\ge 1$. 

Our aim is to state conditions concerning the generator $L_N$ and the
invariant measure $\nu_N$ which ensure that Theorem \ref{teo1} is in
force. We then need good estimates for the mean of entry times. As we
shall see in Section \ref{proof}, in the reversible case, the mean of
an entry time has a simple expression involving capacities, which are
defined as follows. For two disjoint subsets $F$, $G$ of $E_N$ define
$$
\mathcal B_N(F,G) \;:=\; \{f:E_N\mapsto\bb R : \textrm{$f(\eta)=1$ 
$\forall$ $\eta\in F$ and $f(\xi)=0$ $\forall$ $\xi\in G$}\}\;.
$$
We denote by $f_{F,G}:E_N\mapsto\bb R$ the function in $\mc B_N(F,G)$ 
defined as
$$
f_{F,G}(\eta) \;:=\; \bb P^N_{\eta}[H_{F} < H_G]\;.
$$
An elementary computation shows that $f_{F,G}$ solves
\begin{equation*}
\left\{
\begin{array}{ll}
(L_Nf_{F,G})(\eta) =0 & \eta\in E\setminus (F\cup G)\;, \\
f_{F,G}(\eta) = 1 & \eta\in F\;, \\
f_{F,G}(\eta) = 0 & \eta\in G \;.
\end{array}
\right.
\end{equation*}
The maximum principle guarantees that there is a unique solution of
this equation given by $f_{F,G}$.  Let $\nu_N(\,\cdot\,)$ stand for
the expectation with respect to $\nu_N$ and let
$\langle\cdot,\cdot\rangle_{\nu_N}$ stand for the scalar product in
$L^2(\nu_N)$. Denote by $D_N$ the Dirichlet form associated to the
generator $L_N$:
$$
D_N(f)\;:=\;\langle -L_N f,f \rangle_{\nu_N}
$$
for every $f:E_N\mapsto\bb R$. The capacity of two disjoint subsets
$F$, $G$ of $E_N$ is defined as
$$
\Cap_N(F,G) \;:=\; \inf\big\{\,D_N(f) : f\in \mathcal B_N(F,G)\,\big\}\;.
$$
Observe that $\Cap_N(F,G) = \Cap_N(G,F).$ If $F$ or $G$ are equal to
an unitary set $\{\eta\}$ in any notation just introduced, then we
shall write $\eta$ instead of $\{\eta\}$. It is well known that, since
$\nu_N$ is supposed to be reversible,
$$
\Cap_N(F,G)\;=\;D_N(f_{F,G})
$$
(see Theorem II.6.1 in \cite{l}\,).

We now state the hypotheses we shall use to prove that condition {\bf
  (C3)} is in force. Choose for each $x\in S$ and $N\ge 1$, a suitable
site $\zeta^x_N\in \partial_x\Delta_N$, whenever $\partial_x\Delta_N$
is non-empty. We require that
\begin{equation}
\tag*{\bf (H2)} \lim_{N\to\infty} 
\frac{ \nu_N(\Delta_N) }{ \Cap_N (\zeta^x_N,\check {\mc E}^x_N) } 
\;=\;0\;.
\end{equation}

Of course, {\bf (H2)} follows from the stronger condition
\begin{equation}
\tag*{\bf (H2')} \lim_{N\to\infty} \nu_N(\Delta_N)\;=\;0\quad 
\text{and} \quad \inf_{N\ge 1} \Cap_{N}(\zeta^x_N, \check{\mc E}^x_N) 
\;>\; 0\;,
\end{equation}
where the infimum in the right hand side is carried over all $N\ge 1$
such that $\partial_x\Delta_N$ is non-empty. In case that
$\partial_x\Delta_N$ is empty for all $N\ge 1$ then the right hand
side is supposed to be automatically satisfied.

Let $\overline{\mc E_N^x}$ stand for the set $\mc
E^x_N\cup\partial_x\Delta_N$. In addition, for each pair $x,y$ of
different elements of $S$ we denote by $\partial\mc E^{x,y}_N$ the
support of $R^{\mc E}_N(\cdot,y)$ on $\mc E^x_N$:
$$
\partial\mc E^{x,y}_N \;:=\; \{\eta\in \mc E^x_N : R_N^{\mc E}(\eta,y)>0\}.
$$
As our ergodic hypotheses, we shall assume that for any pair $x,y\in
S$ such that $x\neq y$ there holds
\begin{equation*}\tag*{\bf (H3)}
\lim_{N\to\infty}\;\min\big\{\, \Cap_N(\eta,\xi^x_N) : 
\eta\in\overline{\mc E^{x}_N}\,\big\} \;\min\big\{\,\nu_N(\eta) : 
\eta\in\partial\mc E^{x,y}_N\,\big\}\; =\; \infty,
\end{equation*}
where the second minimum is equal to $1$ if $\partial\mc E_N^{x,y}$ is
empty.

We now state our second main theorem that assures a meta-stable
behaviour in the reversible setting. Its proof is postponed to Section
\ref{proof}.

\begin{theorem}
\label{teo2}
A sequence of reversible Markov processes satisfying conditions {\rm
  {\bf (C1)}}, {\rm {\bf (H2)}}, {\rm {\bf (H3)}} is meta-stable in
the sense of Definition \ref{metadef}, with meta-states given by $(\mc
E^x_N)_{x\in S}$ and limit given by the Markov process $(\bb
Q_{\,x})_{x\in S}$ with transition rates $(r(x,y) : x,y\in S)$.
\end{theorem}

\subsection{Meta-stable behavior of condensed zero-range processes}
\label{2.4}

Fix an integer $\kappa\ge2$ and the set of sites
$S=\{1,2,\dots,\kappa\}$. Denote by $\bb N$ the set of nonnegative
integer numbers. For each $N\ge 1$, consider the set of configurations
$$
E_N\;:=\;\{\eta\in\bb N^{S} : \sum_{x\in S}\eta(x)= N\}
$$
where $\eta(x)$ stands for the number of particles at the site $x\in
S$. Fix $\alpha>1$ and let $g:\bb N\backslash\{0\}\mapsto\bb R$ be
given by $g(1)=1$ and
$$
g(n) \;=\; \Big(\frac{n}{n-1}\Big)^{\alpha},\quad\forall n\ge 2
$$
in such a way that $g(1)g(2)\dots g(n)=n^{\alpha}$ for all $n\ge
1$. Notice that $(g(n):n\ge 2)$ is a strictly decreasing sequence and
that $g(n)\to 1$ as $n\to \infty$. On each $E_N$ consider the zero
range process whose generator $L_N$ is given by
$$
L_N f(\eta)\;=\;\sum_{x,y\in S}  g(\eta(x)) 
\bs 1{\{\eta(x)>0\}} \big\{ f(\sigma^{x,y}\eta) - f(\eta) \big\}
$$
for $f:E_N\mapsto\bb R\,.$ Here $\sigma^{x,y}\eta$ is the
configuration obtained from $\eta$ letting a particle jump from $x$ to
$y$:
$$
\sigma^{x,y}\eta(z)\;=\;\left\{\begin{array}{ll}
\eta(x)-1 & \textrm{if $z=x$\;,}\\
\eta(y)+1 & \textrm{if $z=y$\;,}\\
\eta(z) & \textrm{otherwise}\;,
\end{array}
\right.
$$
so $\sigma^{x,y}\eta$ is well defined whenever $\eta(x)>0$. 

For each $N\ge 1$, we denote by $\nu_N$ the unique invariant
probability measure corresponding to $L_N$. It is easy to see that
each $\nu_N$ is a reversible measure and can be written as
$$
\nu_N(\eta) \;=\; \frac{1}{Z_{N,\kappa}}\;\frac{N^{\alpha}}
{\prod_{x\in S} p(\eta(x))}\;,\quad \forall \eta\in E_N
$$
where $p(0)=1$, $p(n)=n^{\alpha}$ if $n\ge 1$ and $Z_{N,\kappa}>0$ is
a normalizing constant. It is not difficult to prove that for each
$\kappa\ge 2$, there exists a constant $\delta_{\kappa} > 0 $ such
that
\begin{equation}\label{zeta}
 \delta_{\kappa} \;<\; Z_{N,\kappa} \;<\; \delta_{\kappa}^{-1} 
\end{equation}
for every $N\ge 1$.

Let $(\ell_N)_{N\ge 1}$ be a sequence of positive integers such that 
$$
\lim_{N\to\infty} N^{-1}\ell_N \;=\; 0 \quad \text{and}\quad 
\lim_{N\to\infty}\ell_N\;=\; \infty \;.
$$ 
This sequence determines the subsets of $E_N$:
\begin{equation*}
\mc E^x_N \;:=\; \{ \eta\in \bb N^{S} ;\; \eta(x)\ge N-\ell_N \}\;,
\quad x\in S\;.
\end{equation*}
Since $N^{-1}\ell_N\to 0$, for $N$ large enough we have $\mc
E^x_N\cap\mc E^y_N=\phi$ for all $x\not = y$ and $E_N\setminus
(\cup_{x\in S}\mc E^x_N) \not = \phi$ . Let $\mc E_N$ stand for
$\cup_{x\in S}\mc E^x_N$ and $\Delta_N$ stand for $E_N\setminus\mc
E_N.$

To establish the time-scale at which the zero-range process exhibit a
non-trivial meta-stable behavior, let $(\theta_N)_{N\ge 1}$ be a
sequence of positive numbers. For each $N\ge 1$, let $\eta^N_t$ stand
for the zero range process $L_N$ speeded-up by $\theta_N$. Let
$(R_N(\eta,\xi) : \eta,\xi\in E_N)$ denote the transition rates of
$\eta^N_t$ :
$$
\theta_N L_N f (\eta) \;=\; \sum_{\xi \in E_N} 
\big\{ f(\xi) - f(\eta)\big\} R_N(\eta,\xi)\;.
$$
We shall choose the time scale by examining condition {\bf
  (C1)}. 

Let $\Cap(\cdot,\cdot)$ be the capacity with respect to the generator
$L_N$. We choose the sequence $(\theta_N)_{N\ge 1}$ as
\begin{equation}\label{teta}
\theta_N \;=\; \frac{1}{\Cap(\mc E^x_N, \check{\mc E}^x_N)}\;,
\quad \text{for $N\ge 1$}\;.
\end{equation}
Since particles jump to any site with equal probability, the right
hand side does not depend on $x\in S$.  We prove in Lemma \ref{capor}
below that there exists a constant $C_0>0$ such that
$$
\frac 1 {C_0}\;<\; \frac{\theta_N}{N^{1+\alpha}}\;<\; C_0
$$
for all $N\ge 1$.

Recall that for $x,y\in S$, $x\not = y$ we denote
$$
r_N(x,y)\;=\; \frac{1}{\nu_N(\mc E^x_N)}\; 
\sum_{\substack { \eta\in\mc E^x_N \\ \xi\in\mc E^y_N}} 
R_N(\eta,\xi)\nu_N(\eta)\;.
$$
Notice that $r_N(x,y)$ does not depend on $x,y\in S$.  By formula (c)
in Lemma \ref{s07},
$$
\sum_{y\in S\setminus\{x\}}r_N(x,y) \;=\; \kappa\;,\quad \forall x\in S\;
$$
and so, $r_N(x,y)=\kappa(\kappa-1)^{-1}$ for all $N\ge 1$ and all
$x,y\in S$, $x\not = y$. In particular, condition {\bf (C1)} is
trivially satisfied.

\begin{theorem}\label{teozr}
  Let $(\theta_N)_{N\ge 1}$ be given by {\rm (\ref{teta})}. If
  $\ell_N\to \infty$ and
\begin{equation}\label{ln}
\lim_{N\to\infty}\frac{(\ell_N)^{\,2\alpha(\kappa-1) + 1}}
{N^{1+\alpha}}\;=\;0\;,
\end{equation}
then the sequence of zero-range processes $\eta^N_t$ is meta-stable in
the sense of Definition \ref{metadef}, with meta-states given by $(\mc
E_N^x)_{x\in S}$. The meta-stable behaviour is described by a Markov
process on $S$ with transition rates $r(x,y)=\kappa(\kappa-1)^{-1}$
for any $x,y\in S$, $x\neq y.$
\end{theorem}

\section{Condensed zero-range processes}

We show in this section that the condensed zero-range processes
introduced in Section \ref{sec0} present a meta-stable behavior in the
sense of Definition \ref{metadef}.  We start with a generalization of
zero-range processes on two sites. Throughout this section we use
systematically the notation introduced in the previous section.

\subsection{Birth and death process}
Fix $a<b$ in $\bb R$ and consider a nonnegative smooth function
$H:[a,b] \to \bb R_+$. Assume that $H$ vanishes only at a finite
number of points denoted by $a_1<a_2< \cdots < a_m$:
\begin{equation*}
H(x) = 0 \quad\text{if and only if}\quad x\in\{a_1, \dots, a_m\}\;.
\end{equation*}
We do not exclude the possibility that $H$ vanishes at the boundary
points $a$, $b$. 

For each $i=1,\dots,m$, assume that there exist a neighborhood
$V_{a_i}$ of $a_i$ and $\alpha_i >0$ such that
\begin{equation*}
  H (x) = |x-a_i|^{\alpha_i} \quad \text{for all $x\in V_{a_i}$}\;,
\end{equation*}
and that $V_{a_i}\cap V_{a_j}=\phi$ for $i\not = j$. Let $\alpha =
\max\{\alpha_i : 1\le i\le m\}$. Assume that $\alpha >1$ and that
there are at least two exponents $\alpha_i$ equal to $\alpha$:
\begin{equation*}
\kappa \;:=\; \big| \{i : \alpha_i = \alpha\}\big| \;\ge\; 2\;,
\end{equation*}
where $|A|$ indicates the cardinality of a finite set $A$. Denote by
$b_1 < b_2 < \cdots <b_\kappa $ the elements of $\{a_1, \dots, a_m\}$
whose associated exponents are $\alpha$.

The definition of the state space of the birth and death process
requires some notation. Fix $N\ge 1$, $1\le i\le m-1$, and let 
\begin{eqnarray*}
\!\!\!\!\!\!\!\!\!\!\!\!\! &&
k^N_{0} \;=\; \min\Big\{ k\ge 0 \;:\; a > a_{1} 
- (k+1)/N \Big\}\; , \\
\!\!\!\!\!\!\!\!\!\!\!\!\! && \quad
k^N_{i} \;=\; \min\Big\{ k\ge 0 \;:\; a_i + (k+1)/N > a_{i+1} 
- (k+1)/N \Big\}\;, \\
\!\!\!\!\!\!\!\!\!\!\!\!\! && \qquad
k^N_{m} \;=\; \min\Big\{ k\ge 0 \;:\; a_m + (k+1)/N > b \Big\}\;.
\end{eqnarray*}
For $1\le i\le m-1$, set 
\begin{eqnarray*}
\!\!\!\!\!\!\!\!\!\!\!\!\! &&
G_{N,0} \;=\; \big\{a_1 - k^N_0/N, \dots, a_1 - 1/N, a_1\big\}\;, \\
\!\!\!\!\!\!\!\!\!\!\!\!\! && \quad 
G_{N,i} \;=\; \big\{a_i +1/N, a_i + 2/N, \dots, a_i + k^N_i/N, 
a_{i+1} - k^N_i/N, \dots, a_{i+1}\big\}\;, \\
\!\!\!\!\!\!\!\!\!\!\!\!\! && \qquad 
G_{N,m} \;=\; \big\{a_m+1/N, a_m+2/N, \dots, a_m +k^N_m/N\big\}
\end{eqnarray*}
and let the state space $E_N$ be $\bigcup_{i=0}^m G_{N,i}$. Note that
the exact definition of $E_N$ is not important for the meta-stability
behavior discussed in this section. The elements of $E_N$ are denoted
by the letters $x$, $y$, $z$. Two points $x<y$ are said to be
neighbors in $E_N$ if there is no $z$ in $E_N$ such that $x<z<y$.

Let $\nu_N$ be the probability measure on $E_N$ defined by
\begin{equation*}
\nu_N(x) \;=\; 
\left\{
\begin{array}{ll}
\displaystyle{
\frac 1{Z_N} \frac 1{H(x)} }& \text{if } x\not\in \{a_1, \dots,
a_m\}\;, \\
\displaystyle{
\frac 1{Z_N} N^{\alpha_i}} & \text{if } x= a_i \text{ for some $1\le
  i\le m$} \;.
\end{array}
\right.
\end{equation*}
In this formula $Z_N$ is a normalizing constant. An elementary
computation shows that
\begin{equation}
\label{ef04}
\lim_{N\to\infty} \frac{Z_N}{N^\alpha} \;=\;
\sum_{i=1}^\kappa \Big\{ 1 + \sigma_i \sum_{k\ge 1} \frac
1{k^\alpha}\Big\} \;.
\end{equation}
where $\sigma_i=1$ if $b_i\in \{a,b\}$ and $\sigma_i=2$ otherwise. In
particular, if we denote
\begin{equation*}
 m(b_i) \;=\; 1 + \sigma_i \sum_{k\ge 1} \frac 1{k^\alpha}\;, 
\quad \text{for } i=1, \dots, \kappa
\end{equation*}
then
\begin{equation}
\label{g01}
 \lim_{N\to\infty} \nu_N(b_i) \;=\; 
\Big\{ \sum_{i=1}^{\kappa} m(b_i) \Big\}^{-1} \;>\; 0
\end{equation}
for every $1\le i\le \kappa$.

We now fix a convenient sequence of neighborhoods around each point
$b_i$, $i=1,\dots,\kappa$. Let $(\ell_N)_{N\ge 1}$ be a sequence of
positive integers, both small for the macroscopic scale and large for
the microscopic scale:
\begin{equation*}
 \lim_{N\to\infty} N^{-1}\ell_N \;=\; 0 \quad 
\text{and}\quad \lim_{N\to\infty}\ell_N\;=\; \infty \;.
\end{equation*}
For each $b_i$, $1\le i\le \kappa$, we define
$$
\mc E^i_N \;:=\; E_N \cap \Big[\,b_i - \frac{\ell_N}{N}\,,\, 
b_i + \frac{\ell_N}{N}\,\Big].
$$
Notice that, because $N^{-1}\ell_N \to 0$, for $N$ large enough we
have $\mc E^i_N \subseteq V_{b_i}$ for every $1\le i \le \kappa$. In
particular, for $N$ large enough, $\mc E^i_N\cap\mc E^j_N=\phi$ for
all $i\not = j$. Moreover, since $\ell_N\to\infty$,
\begin{equation}\label{nuE}
 \lim_{N\to\infty} \nu_N(\Delta_N) \;=\; 0 \quad\text{and} 
\quad \lim_{N\to\infty}\nu_N(\mc E^i_N) \;=\; \frac{m(b_i)}
{\sum_{j=1}^{\kappa} m(b_j)}
\end{equation}
for all $1\le i\le \kappa$, where $\Delta_N := E_N \setminus
(\cup_{i=1}^{\kappa} \mc E^i_N)$.

In view of the theory presented in Section \ref{sec1}, this
concentration of the measure on small macroscopic subsets suggests
that a Markov process on $E_N$, reversible with respect to $\nu_N$, is
a good candidate to exhibit a meta-stable behavior with $\{b_1, \dots,
b_\kappa\}$ as meta-stable points. 

Fix a positive function $\lambda: [a,b]\to \bb R_+$ bounded above and
below by a strictly positive constant:
\begin{equation*}
0\;<\; \delta \;\le\; \lambda (x) \;\le\; \delta^{-1}\;.
\end{equation*}
This assumption is not really needed but we do not seek optimal
assumptions in this subsection. Consider a birth and death process
$\{Z^N_t : t\ge 0\}$ on $E_N$ with rates given by
\begin{equation*}
R_N(x,y) \;=\; 
\left\{
\begin{array}{ll}
{\displaystyle
N^{1+\alpha} \, \lambda(x)} & \text{if $x>y$,}  \\
{\displaystyle
N^{1+\alpha} \, \lambda(y) \, \frac{\nu_N(y)}{\nu_N(x)}} & 
\text{if $x<y$}\;,
\end{array}
\right.
\end{equation*}
provided $x$ and $y$ are neighbors in $E_N$. The process is of course
reversible with respect to the measure $\nu_N$. As we shall see in the
proof of the next result $N^{1+\alpha}$ is the correct scale to
observe a non-trivial meta-stable behaviour.

\begin{theorem}
\label{bd}
Assume that $(\ell_N)_{N\ge 1}$ is chosen such that $\ell_N\to\infty$
and
\begin{equation}\label{ell}
\lim_{N\to\infty}\frac{(\ell_N)^{2\alpha+1}}{N^{\alpha+1}} \;=\; 0\;.
\end{equation}
Then, $(Z^N_t)_{N\ge 1}$ is meta-stable in the sense of Definition
\ref{metadef}, with meta-states given by $(\mc E^i_N)_{1\le i\le
  \kappa}$. The meta-stable behavior is described by a Markov process
on $\{1,\dots,\kappa\}$ with rates given by
\begin{eqnarray*}
\!\!\!\!\!\!\!\!\!\!\!\!\!\!\!\!
&& r(i, i+1) \;=\; \frac 1{m(b_i)} \, \frac
1{\int_{b_i}^{b_{i+1}} \{H(u)/\lambda(u)\} \, du}\; , \\
\!\!\!\!\!\!\!\!\!\!\!\!\!\!\!\!
&& \quad r(i+1, i) \;=\; \frac 1{m(b_{i+1})} \, \frac
1{\int_{b_i}^{b_{i+1}} \{H(u)/\lambda (u)\} \, du}\;,
\end{eqnarray*}
for $1\le i<\kappa$.
\end{theorem}

\begin{proof}
We shall verify each hypothesis of Theorem \ref{teo2}. Let $\xi^i_N =
b_i$ for $1\le i\le \kappa$. The first condition of {\bf (H2')} has
been derived in \eqref{nuE}. 

To check the second condition of {\bf (H2')} and {\bf (H3)}, we take
advantage from the one-dimensional setting to get explicit expressions
for capacities.  Let $x^N<y^N$ be points in $E_N$. Recall that
$\Cap(x^N,y^N)=D_N(f_{x^N,y^N})$ where $f_{x^N,y^N}:E_N\mapsto \bb R$
solves the equation $L_Nf_{x^N,y^N}(z)=0$ for $z\not \in \{x^N,y^N\}$
with boundary conditions $f_{x^N,y^N}(x^N)=1$ and
$f_{x^N,y^N}(y^N)=0$. An elementary computation gives that $f(z)=1$
for $z\le x^N$, $f(z)=0$ for $z\ge y^N$ and
\begin{equation*}
f(z+1/N) - f(z) \;=\; \frac{ \big\{ \nu_N(z) R_N(z,z+1/N) \big\}^{-1}}
{ \sum_{z=x^N}^{y^N-1/N} \big\{ \nu_N(z) R_N(z,z+1/N) \big\}^{-1}}\;
\end{equation*}
for $z\in E_N \cap [x^N,y^N[\,$. Hence,
\begin{equation}\label{cap1}
\Cap(x^N,y^N) \;=\; \frac{1}{ \sum_{z=x^N}^{y^N-1/N} 
\big\{ \nu_N(z) R_N(z,z+1/N) \big\}^{-1}}\;.
\end{equation}
In last two formulae, there is a slight abuse of notation since $E_N$
is not the set $\{z/N : z\in\bb Z \cap [aN, bN]\}$, but the meaning is
clear. In particular, it follows from (\ref{cap1}) that, if $x^N\to
a'$ and $y^N\to b'$ for some $0\le a' < b'\le 1$ then
\begin{equation}\label{capl}
 \lim_{N\to\infty}\Cap(x^N,y^N) \;=\; \frac{1}{\sum_{i=1}^{\kappa}
   m(b_i)}\;\Big\{\;\int_{a'}^{b'} \{H(u)/\lambda (u)\}
\, du \;\Big\}^{-1}\;>\;0\;.
\end{equation}

We are now in a position to check the second condition of hypothesis
{\bf (H2')}. For each $1\le i\le \kappa$ and $N\ge 1$, we choose the
point $c^N_i\in\partial_i\Delta_N$ as $c^N_i=b_i - (\ell_N+1)/N$ if
$b_i=b$ and $c^N_i=b_i + (\ell_N+1)/N$ otherwise. Taking advantage of
the one-dimensional setting, it is not difficult to show that
$\Cap(c^N_i , \check{\mc E}^i_N)$ is equal to
$$
\mb{1}\{i<\kappa\}\; \Cap(c^N_i,b_{i+1} - \ell_N/N) \;+\; 
\mb{1}\{1<i\}\; \Cap(c^N_i,b_{i-1} + \ell_N/N)\;,
$$
for $1\le i\le \kappa$. Therefore, hypotheses {\bf (H2')} follows from
(\ref{capl}).

Recall definition of subsets $\partial\mc E^{i,j}_N$, $\overline{\mc
  E^i_N}$ and $\partial_i\Delta_N$ of $E_N$ introduced in the previous
section.  To check hypotheses {\bf (H3)} notice that in our model,
$$
\overline{\mc E^i_N} \;=\; E_N\cap [\,b_i-(\ell_N+1)/N\,,\,
b_i+(\ell_N+1)/N\,] \quad \text{for $1\le i\le \kappa$}\;.
$$
Observe that for any $1\le i \le \kappa$ and $N$ large enough,
$H(x)=|x-b_i|^{\alpha}$ for all $x\in\overline{\mc E^i_N}$. In
consequence, by using (\ref{cap1}), it is easy to see that there
exists a positive constant $C_0$ depending only on $\lambda$ such that
\begin{equation}
\Cap(x,b_i) \;\ge\; \frac{C_0 N^{2\alpha + 1}}
{Z_N(\ell_N+1)^{\alpha +1}}\;,
\end{equation}
for any $1\le i\le\kappa$ and $x\in\overline{\mc E^i_N}$. In
particular, by (\ref{ef04}),
$$
\lim_{N\to\infty}\min \big\{\, \Cap(x,b_i) : x \in 
\overline{\mc E^i_N} \,\big\}\;=\;\infty\;
$$
since $N^{-1}\ell_N\to 0$. Moreover, for any two points $i\not =j$ in
$\{1,\dots,\kappa\}$, $\partial\mc E^{i,j}_N = \{b_i +
(j-i)\ell_N/N\}$ if $|i-j|=1$ and $\partial\mc E^{i,j}_N=\phi$
otherwise. Therefore, in view of assumption (\ref{ell}), an elementary
calculation shows that hypotheses {\bf (H3)} is in force.

It remains to check hypotheses {\bf (C1)}. Notice that, in our model,
$r_N(i,j)=0$ for $|i-j|>1$. For $G_1,G_2$ disjoint subsets of $\mc
E_N$, let $r_N(G_1,G_2)$ be given by
$$
r_N(G_1,G_2)\;:=\; \frac{1}{\nu_N(G_1)}
\sum_{\substack{x\in G_1 \\ y \in G_2}}R^{\mc E}_N(x,y)\nu_N(x)\;
$$
in such a way that, for $i\not =j$, $r_N(\mc E^i_N, \mc E^j_N)$
coincides with $r_N(i,j)$. Fix an arbitrary $1\le i<\kappa$. Observe
that letting $G_1=\cup_{j\le i} \mc E^j_N$, and $G_2=\cup_{j>i}\mc
E^j_N$ we have
\begin{eqnarray*}
\nu_N(\mc E^i_N)\,r_N(i,i+1) &=& \nu_N(b_i+\frac{\ell_N}{N})\,
R^{\mc E}_N(b_i + \frac{\ell_N}{N}, b_{i+1} - \frac{\ell_N}{N}) \\
&=& \nu_N (G_1)\, r_N(G_1,G_2)\;.
\end{eqnarray*}
Therefore, by using formula (c) of Lemma \ref{s07} for $r_N(G_1,G_2)$,
\begin{eqnarray*}
 r_N(i,i+1) &=& \frac{\nu_N(G_1)\,r_N(G_1,G_2)}{\nu_N(\mc E^i_N)} 
\;=\; \frac{\Cap(G_1,G_2)}{\nu_N(\mc E^i_N)}  \\
&=& \frac{\Cap(b_i+\ell_N/N,b_{i+1} - \ell_N/N)}{\nu_N(\mc E^i_N)} \;\cdot
\end{eqnarray*}
Analogously, we obtain
$$
r_N(i,i-1)\;=\;\frac{\Cap(b_{i}-\ell_N/N , b_{i-1} + \ell_N/N)}
{\nu_N(\mc E^{i}_N)} \;
$$
for any $1<i\le \kappa$. Therefore, by (\ref{capl}) and (\ref{nuE}),
$$
\lim_{N\to\infty} r_N(i,i+1) \;=\; \frac 1{m(b_i)} \, 
\Big\{ \,\int_{b_i}^{b_{i+1}} \{H(u)/\lambda(u)\} \, du \,\Big\}^{-1}\;
$$
and
$$
\lim_{N\to\infty} r_N(i+1,i) \;=\; \frac 1{m(b_{i+1})} \, 
\Big\{\,\int_{b_i}^{b_{i+1}} \{H(u)/\lambda(u)\} \, du \,\Big\}^{-1}\;
$$
for any $1\le i < \kappa.$ We are done.
\end{proof}

\subsection{Condensed zero-range processes}

We prove in this subsection Theorem \ref{teozr}. The case $\kappa =2$
is a particular case of the birth and death processes considered in
the previous section. Indeed, since the total number of particles is
settled to be equal to $N$, the number of particles on the first site,
$\eta(1)$, is a Markov process on $\{0, 1, \dots, N\}$. The density of
particles speeded up by $N^{1+\alpha}$, $\{ \eta_{t N^{1+\alpha}} (1)
: t\ge 0\}$, corresponds to the birth and death process on $\{0, 1/N,
\dots, 1\}$ associated to the functions $H(x) = x^\alpha
(1-x)^\alpha$, $\lambda_N(x) = (x/x-N^{-1})^\alpha \mb 1\{x\ge 2/N\} +
\mb 1\{x=1/N\}$. In contrast with the set-up of the previous section,
$\lambda_N$ depends on $N$, but all the arguments go through. Observe
that $\lambda_N$ converges to $1$ as $N\uparrow\infty$.

By Theorem \ref{bd}, on the scale $N^{1+\alpha}$, the zero-range
process on two sites has a meta-stable behavior with meta-states $\mc
E^1_N = \{\eta : \eta(1) \ge N - \ell_N\}$, $\mc E^2_N = \{\eta :
\eta(2) \ge N - \ell_N\}$, and asymptotic Markov evolution with rates
$r(1,2) = r(2,1) = (m \int_0^1 u^\alpha (1-u)^\alpha du)^{-1}$, where
$m = 1 + \sum_{k\ge 1} k^{-\alpha}$.
\medskip

We now turn to the general case. Recall the definitions introduced in
Subsection \ref{2.4}. We first estimate the order of $(\theta_N)_{N\ge
  1}$. For each $x\in S$ and $N\ge 1$, let the configuration
$\xi^x_N\in\mc E^x_N$ be given by $\xi^x_N(x)=N$. Denote by $D_N$ the
Dirichlet form associated to the generator $L_N$.

In the next lemma and below, we adopt the convention that $C_0$ stands
for a positive finite constant depending only on $\kappa$ and $\alpha$
whose value may change from line to line.

\begin{lemma}\label{capor}
There exists a constant $C_0>0$ such that
$$
\frac 1 {C_0}\;<\; \frac{\theta_N}{N^{1+\alpha}}\;<\; C_0
$$
for all $N\ge 1$.
\end{lemma}

\begin{proof}
For each $N\ge 1$, we fix the set $I_N:=\{0,1,\dots,N\}$ and the
measure
$$
\mu_N(n)=\frac{1}{Z_{N,2}}\;\frac{N^{\alpha}}{p(N-n)p(n)}\;,\quad n\in
I_N\;,
$$
so that, if $\kappa=2$, $\nu_N(\eta)=\mu_N(\eta(x))$ for any $x\in
\{1,2\}$. 

For $\kappa \ge 2$, given $x,y\in S$, $x\neq y$ and $N\ge 1$, a simple
calculation shows that for any $\eta'\in E_N$ such that $\eta'(x)=N-n$
and $\eta'(y)=n$ for some $n\in I_N$ we have
\begin{equation*}
\frac{Z_{N,2}}{Z_{N,\kappa}}\mu_N(n) \;=\; 
\nu_N(\eta') \;\le\; \nu_N \big(\{\eta : \eta(x)=N-n\}\big) 
\;=\; \frac{Z_{N,2} Z_{n,\kappa-1}}{Z_{N,\kappa}} \mu_N(n)\;.
\end{equation*}
In particular, there exists a constant $C_0 >0$ such that
\begin{equation}\label{numu}
\frac{1}{C_0} \;<\; \frac{\nu_N(\eta')}{\mu_N(n)} 
\;\le\; \frac{\nu_N \big(\{\eta : \eta(x)=N-n\}\big)}{\mu_N(n)} 
\;<\; C_0\;,
\end{equation}
for all $N\ge 1$ and $n\in I_N$.

Fix $x\in S$ and $N\ge 1$. To get a lower bound for $\Cap_N(\mc
E^x_N,\check{\mc E}^x_N)$, fix an arbitrary $f\in\mc B_{N}(\mc
E^x_N,\check{\mc E}^x_N)$ and choose a site $y\in
S\backslash\{x\}$. Consider the path $\eta^{(j)},$ $j=0,1,\dots,N$
from $\eta^{(0)}=\xi^x_N$ to $\eta^{(N)}=\xi^y_N$ satisfying
$\eta^{(j+1)}=\sigma^{x,y}\eta^{(j)}$ for $j=0,1,\dots N-1.$ By using
Cauchy-Schwarz inequality and (\ref{numu}) we obtain
\begin{eqnarray}
1 &=& \Big\{\; \sum_{j=0}^{N-1} \big[\, f(\eta^{(j+1)}) -
f(\eta^{(j)}\,)\big] 
\; \Big\}^2 \nonumber \\
&\le& D_N(f)\, \sum_{j=0}^{N-1} \frac{1}{\nu_N(\eta^{(j)})
g(\eta^{(j)}(x))}\, \label{dinu} \\
&\le& D_N(f) \sum_{j=0}^{N-1}\frac{C_0}{\,\mu_N(j)}\;\cdot 
\nonumber
\end{eqnarray}
Therefore,
\begin{equation*}
\Cap_N(\mathcal E^x_N,\check{\mc E}_N^x) 
\;\ge\; C_0\; \Big\{ \sum_{j=0}^{N-1} \frac{1}{\mu_N(j)} \Big\}^{-1}
\;\ge\; \frac{C_0}{N^{\alpha +1}}\;\cdot
\end{equation*}
To get a lower bound, consider the function $\varphi_N:I_N\mapsto
\bb R$ defined as $\varphi_N(n)=0$ for $n\le \ell_N$, $\varphi_N(n)=1$
for $n\ge N - \ell_N$ and
$$
\varphi_{N}(n) - \varphi_{N}(n-1) \;=\; \Big\{ \sum_{k=\ell_N+1}^{N-\ell_N}
\frac{1}{\mu_N(k)g(k)} \Big\}^{-1} \frac{1}{\mu_N(n)g(n)}
$$
for $\ell_N +1 \le n \le N - \ell_N$. Then, define the function
$f_N:E_N\mapsto \bb R$ by $f_N(\eta):=\varphi_{N}(\eta(x))$. Hence,
$f_N\in\mc B_N(\mc E^x_N,\check{\mc E}^x_N)$ and so $\Cap_N(\mc
E^x_N,\check{\mc E}^x_N)$ is less than or equal to
\begin{equation*}
 D_N(f_N) \;=\; \sum_{n=1}^{N}\big[ \varphi_{N}(n) - \varphi_{N}(n-1)
 \big]^2 (\kappa-1)g(n)\nu_N\big( \{\eta : \eta(x)=n \} \big)\;.
\end{equation*}
By using (\ref{numu}) and definition of $\varphi_N$ we obtain
$$
\Cap_N(\mc E^x_N,\check{\mc E}^x_N) \;\le\;
\frac{C_0}{N^{\alpha+1}}\;\cdot
$$
\end{proof}

We now turn to the 

\noindent \textbf{Proof of Theorem \ref{teozr}.}  Since the process is
reversible, by Theorem \ref{teo2} and by the definition of $\theta_N$,
we just need to verify conditions {\bf (H2)}, {\bf (H3)}.

Since $\ell_N\to\infty$, by (\ref{ln}), $N^{-1}\ell_N \to 0$. In
particular, an elementary computation shows that
$$
\lim_{N\to\infty}\nu_N(\Delta_N) \;=\; 0\quad \text{and}\quad 
\lim_{N\to\infty}\nu_N(\mc E^x_N) \;=\; 1/\kappa
$$ 
for every $x\in S$. The first identity corresponds to the first
assumption of {\bf (H2')}.

For the second claim in {\bf (H2')}, notice that in this setting
$$
\partial_x\Delta_N \;=\; \{\eta\in E_N : \eta(x)=N-\ell_N-1\}\;.
$$
For each $x\in S$ and $N\ge 1$ we choose the configuration
$\zeta^x_N\in\partial_x\Delta_N$ such that $\zeta^x_N(x)=N-\ell_N-1$
and $\zeta^x_N(y)=\ell_N+1$ for some $y\neq x.$ It is not difficult to
adapt the proof of the lower bound for $\Cap_N(\mc E^x_N,\check{\mc
  E}^x_N)$ to show that
$$
\Cap_N(\zeta^x_N,\check{\mc E}^x_N) \;\ge\; \frac{C_0}{N^{\alpha+1}}\;.
$$
Therefore, the second condition in {\bf (H2')} follows from Lemma
\ref{capor} if one recalls that $\Cap$ stands for the capacity
associated to the generator $L_N$ which has not been speeded up.

Finally, we check hypothesis {\bf (H3)}. In this setting,
$$
\partial\mc E^{x,y}_N \;=\; \{\eta\in \mc E_N : \eta(x)=N-\ell_N\}
$$
and
$$
\overline{\mc E^x_N} \;=\; \{\eta\in \mc E_N : \eta(x) \ge N-\ell_N -1\}\;
$$
for any $x, y\in S$, $x\not = y$ and $N\ge 1$. Fix $x\in S$. It is not
difficult to prove that for any $\eta\in \overline{\mc E^x_N}$ and $N$
large enough
\begin{equation}\label{minu}
\nu_N(\eta) \;\ge\; \frac{1}{Z_{N,\kappa}}\;
\frac{1}{(\ell_N+1)^{\alpha(\kappa-1)}}\;\cdot 
\end{equation}
By similar reasons,
\begin{equation}\label{e1}
\min\{ \nu_N(\eta) : \eta \in \partial\mc E^{x,y}_N \} 
\;\ge\; \frac{C_0}{(\ell_N)^{\alpha (\kappa-1)}}\;\cdot
\end{equation}
Fix an arbitrary $\eta\in \overline{\mc E^x_N}$. Consider a path 
$$
\eta^{(j)}\in \overline{\mc E^x_N},\; j=0,1,\dots,p
$$
such that $\eta^{(0)}=\eta$, $\eta^{(p)}=\xi^x_N$, $p\le \ell_N$ and
$R_N(\eta^{(j)},\eta^{(j+1)})>0$ $\forall$ $0\le j< p$. Let $f$ be an
arbitrary function in $\mc B_N(\eta,\xi^x_N)$. By using Cauchy-Schwarz
as in (\ref{dinu}) we obtain
\begin{equation}\label{dnu}
D_N(f) \;\ge\; \Big\{ \sum_{j=0}^{p-1}\frac{1}{\nu_N(\eta^{(j)})}  
\Big\}^{-1}\;.
\end{equation}
So, it follows from (\ref{minu}) and (\ref{dnu}) that
\begin{equation}\label{e2}
\min\big\{\, \Cap_N(\eta,\xi^x_N) : \eta\in \overline{\mc E^x_N}\,\big\} 
\;\ge\; \frac{C_0 }{(\ell_N)^{\alpha(\kappa -1) +1}}
\end{equation}
for any $x\in S$. Finally, putting together (\ref{e1}) and (\ref{e2}),
and using Lemma \ref{capor}
$$
\theta_N \min\big\{\, \Cap_N(\eta,\xi^x_N) : 
\eta\in \overline{\mc E^x_N}\,\big\}\; 
\min\{ \nu_N(\eta) : \eta \in \partial\mc E^{x,y}_N \} 
\;\ge\; \frac{C_0 N^{\alpha +1}}{(\ell_N)^{2\alpha(\kappa-1) + 1}}\;\cdot
$$
Hence, hypotheses {\bf (H3)} follows from assumption (\ref{ln}).
\qed


\section{Meta-stability}\label{proof}

In this section we prove Theorem \ref{teo1}, Theorem \ref{teo2} and
Proposition \ref{corol}.

\subsection{Proof of Theorem \ref{teo1}}
We divide the proof of Theorem \ref{teo1} in two parts: proof of
condition (A) and proof of condition (B) in Definition \ref{metadef}.

We start with a replacement lemma which is the key point in the proof
of condition (A). For a function $V_N: \mc E_N \mapsto \bb R$, denote
by $\hat V_N$ its $\nu_N$-conditional expectation given the
$\sigma$-algebra generated by the sets $(\mc E^x_{N})_ {x\in S}$ :
\begin{equation*}
\hat V_N (\eta) \;=\; \frac 1{\nu_N(\mc E^x_{N})} 
\sum_{\xi\in \mc E^x_{N}} V(\xi) \;\nu_N(\xi)
\end{equation*}
for any $x\in S$ and $\eta$ in $\mc E^x_{N}$. Let us also denote
$$
\sigma_N(x)\;:= \; \max_{\eta\in\mc E^x_N}
\bb E^{\mc E}_{\eta}[H_{\xi^x_N}]\;.
$$

\begin{lemma}
\label{s03}
Fix a sequence of functions $V_N: \mc E_N \mapsto \bb R$ such that
$$
\lim_{N\to\infty}\sigma_N(x) \max_{\eta\in\mc E^x_N}
\big|V_N(\eta)\big| \;=\; 0\;.
$$
for any $x\in S$. Then, for any sequence $(\eta_N)_{N\ge 1}$
in $\mc E_N$,
\begin{equation*}
\lim_{N\to\infty} \bb E^{\mc E}_{\eta_N} \Big[ \int_0^t  
\big\{ V_N (\eta^{\mc E_N}_s) - \hat V_N (\eta^{\mc E_N}_s) \big\} 
\, ds \Big] \;=\; 0\;.
\end{equation*}
\end{lemma}

\begin{proof}
Assume without loss of generality that $V_N$ vanishes outside some
subset $\mc E^x_{N}$, i.e. $V_N(\eta) = 0$ for $\eta \in \check{\mc
  E}^x_{N}$.  By Corollary \ref{s09}, the expectation appearing in the
statement of the lemma is bounded by
\begin{equation*}
2 \sigma_N(x)\Vert V_N - \hat V_N \Vert_\infty\;.
\end{equation*}
Since $\Vert \hat V_N \Vert_\infty \le \Vert V_N \Vert_\infty$ this
expression vanishes as $N\uparrow\infty$ by assumption.
\end{proof}
\medskip

\noindent\textbf{Condition (A).}
Fix $x\in S$ and a sequence $(\eta_N)_{N\ge 1}$ such that $\eta_N\in
\mc E^x_N$ for all $N\ge 1$. The convergence of the sequence $\{\bb
Q^N_{\,\eta_N}\}_{N\ge 1}$ stated in condition (A), follows from
tightness and uniqueness of limit points. We first examine the
tightness.

Denote by $(\bb P^{\mc E}_{\eta})_{\eta\in \mc E_N}$ the trace of the
Markov process $\eta^N_t$ on $\mc E_N$, and denote by $\bb
E_{\,\eta}^{\mc E}$ the expectation with respect to $\bb P^{\mc
  E}_{\eta}$.

\begin{lemma}
\label{s11}
The sequence $(\bb Q^N_{\,\eta_N})_{N\ge 1}$ is tight provided conditions
{\rm \bf (C1)}, {\rm \bf (C2)} hold.
\end{lemma}

\begin{proof} 
For each $T>0$, let $\mf T_T$ denote the set of all stopping times
bounded by $T$. By Aldous criterium (see Theorem 16.10 in \cite{b}) we
just need to show that
\begin{equation}
\label{aldous}
\lim_{\delta\downarrow 0}\lim_{N\to\infty}\sup_{\theta\le\delta}
\sup_{\tau\in\mf T_T}\bb P^{\mc E}_{\eta_N} \big[ \;
| X^N_{\tau + \theta} - X^N_{\tau}| > \epsilon \;\big] \;=\; 0
\end{equation}
for every $\epsilon>0$ and $T>0$.

Recall that we denote by $L^{\mc E}_N$ the generator of the trace
process $\eta^{\mc E_N}_t$. Denote by $M^N_t$ the martingale defined by
\begin{equation*}
  M^N_t \;=\; X^N_t \;-\; X^N_0 \;-\; \int_0^t L^{\mc E}_N X^N_s ds\; .
\end{equation*}
To prove tightness of the process $X^N_t$, it is therefore enough to
show that \eqref{aldous} holds with the difference $X^N_{\tau+\theta}
- X^N_{\tau}$ replaced by $M^N_{\tau+\theta} - M^N_{\tau}$ and by
$\int_\tau^{\tau+\theta} L^{\mc E}_N X^N_s ds$.

Consider the integral term. By Chebychev inequality and by the strong
Markov property, we need to prove that
\begin{equation*}
\lim_{\delta\downarrow 0}\, \lim_{N\to\infty}\, \sup_{\theta\le\delta}
\, \max_{\xi \in\mc E_N} \bb E^{\mc E}_{\xi} \Big[ \;
\Big| \int_0^\theta L^{\mc E}_N X^N_s ds \Big | \;\Big] 
\;=\; 0\;.
\end{equation*}
An elementary computation shows that
\begin{equation*}
L^{\mc E}_N \Psi_N (\eta) \;=\; \sum_{x=1}^\kappa 
\{x - \Psi_N(\eta)\} \, R^{\mc E}_N(\eta, x) 
\end{equation*}
for each $\eta$ in $\mc E_N$. The proof is thus reduced to the claim
\begin{equation*}
\lim_{\delta\downarrow 0}\, \lim_{N\to\infty}\,
\max_{\xi \in\mc E^y_N} \, \bb E^{\mc E}_{\xi} \Big[  \int_0^\delta
R^{\mc E}_N(\eta_s, x)  ds \Big] \;=\; 0
\end{equation*}
for all $x$, $y$ in $\{1, \dots, \kappa\}$. 

Fix $x$, $y$. In the previous expectation add and subtract the
$\nu_N$-conditional expectation of $R^{\mc E}_N(\eta, x)$ with respect
to $\mc E^y_N$, denoted by $r_N(y,x)$.  By condition {\bf (C1)}, the
expectation involving $r_N$ is bounded by $C_0 \delta$, which vanishes
as $\delta\downarrow 0$. On the other hand, by condition {\bf (C2)},
$\max_{\eta\in\mc E^y_N} R_N^{\mc E}(\eta,x) \max_{\eta\in\mc E^y_N}
\bb E^N_{\,\eta}[H_{\xi^y_N}]$ vanishes as $N\uparrow\infty$. Since
$\bb E^{\mc E}_{\,\eta}[H_{\xi^y_N}] \le \bb
E^N_{\,\eta}[H_{\xi^y_N}]$, we also have that $\max_{\eta\in\mc E^y_N}
R_N^{\mc E}(\eta,x) \max_{\eta\in\mc E^y_N} \bb E^{\mc
  E}_{\,\eta}[H_{\xi^y_N}]$ vanishes as $N\uparrow\infty$.  In
particular, by Lemma \ref{s03}, the expectation involving the
difference $R^{\mc E}_N(\eta, x) - r_N(y,x)$ vanishes as
$N\uparrow\infty$ for every $\delta>0$.

We now turn to the martingale part $M^N_t$, whose quadratic variation
is given by the time integral of $L^{\mc E}_N \Psi_N(X^N_t)^2 - 2
\Psi_N(X^N_t) L^{\mc E}_N \Psi_N(X^N_t)$. An elementary computation
shows that this expression is equal to
\begin{equation*}
\sum_{x=1}^\kappa \{x - \Psi_N(\eta)\}^2  \, R^{\mc E}_N(\eta, x)\;.
\end{equation*}

Denote by $\< M^N\>_t$ the quadratic variation of the martingale
$M^N$. By the explicit formula for the quadratic variation, by
Chebychev inequality and by the strong Markov property,
\begin{eqnarray*}
\bb P^{\mc E}_{\eta_N} \big[ \, \big| M^N_{\tau+\theta} 
-  M^N_{\tau} \big| \; >\; \epsilon \big] &\le& 
\frac 1{\epsilon^2} \, \bb E^{\mc E}_{\eta_N}
\big[ \<M^N\>_{\tau+\theta} - \<M^N\>_\tau \big] \\
&\le& \frac {C_0}{\epsilon^2} \, \max_{\xi\in \mc E_N}
\bb E^{\mc E}_{\xi} \big[  \int_0^\delta
R^{\mc E}_N(\eta_s, x)  ds \big]
\end{eqnarray*}
It remains to repeat the arguments presented for the intgral term of
the decomposition.
\end{proof}
\medskip

Now we turn to prove the uniqueness of limit points. Assume without
loss of generality that the sequence of probability measures $\bb
Q^N_{\,\eta_N}$ converges to a measure $\bb Q\,$. Recall definition of
$(r_N(x,y) : x,y\in S)$ and their limits $(r(x,y) : x,y\in S)$ in
condition {\bf (C1)}. Denote by $\mf L_N$ and $\mf L$ the Markov
generators on the state space $S=\{1,\dots,\kappa\}$ given by
$$
(\mf L_NF)(x) \;=\; \sum_{y\in S\setminus\{x\}} \{F(y) - F(x)\} r_N(x,y)
$$
and
$$
(\mf LF)(x) \;=\; \sum_{y\in S\setminus\{x\}} \{F(y) - F(x)\}\, r(x,y)\;.
$$
For $t\ge 0$, let $X_t$ denote the projection $D(\bb R_+,S)\mapsto
S$. We shall prove in Lemma \ref{s02} below that $\bb Q\,$ solves the
martingale problem associated to the generator $\mf L$. It is well
known that this property together with the distribution of $X_0$
characterize the measure $\bb Q\,$.

\begin{lemma}
\label{s02}
Assume conditions {\bf (C1)} and {\bf (C2)}. Then, under $\bb Q\,$,
$X_0 = x$ and
\begin{equation}
\label{f04}
M_t \;=\; F(X_t) \;-\; F(X_0) \;-\; \int_0^t \mf L F (X_s) \, ds
\end{equation}
is a martingale for any function $F: S \mapsto \bb R$.
\end{lemma}

\begin{proof}
Fix $0\le s< t$, a function $F: S \mapsto \bb R$ and a bounded
function $U:D(\bb R_+,S)\mapsto \bb R$ depending only on $\{X_r : 0\le
r\le s\}$, continuous for the Skorohod topology.  Denote by $\bb
E_{\bb Q}$ and $\bb E_{\bb Q^N_{\eta_N}}$ the expectation with respect
to $\bb Q$ and $\bb Q^N_{\,\eta_N}$, respectively. We shall prove that
\begin{equation}
\label{f11}
\bb E_{\bb Q}[M_t U] \;=\; \bb E_{\bb Q}[M_s U]\;.
\end{equation}

Recall that $L^{\mc E}_N$ denotes the generator of the trace process
$\eta^{\mc E_N}_t$. For $N\ge 1$, consider the $\bb P^{\mc
  E}_{\eta_N}$-martingale $M^N_t$ defined by
\begin{equation*}
M^N_t \;=\; F(X^N_t) \;-\; F(X^N_0) \;-\; \int_0^t L^{\mc E}_N
(F\circ\Psi_N)(\eta^{\mc E_N}_s) \, ds\;.
\end{equation*}
As $M^N_t$ is a martingale, $\bb E^{\mc E}_{\eta_N}[M^N_t U] 
= \bb E^{\mc E}_{\eta_N}
[M^N_s U]$ so that
\begin{equation*}
\bb E^{\mc E}_{\eta_N} \Big [ U \Big\{ F(X^N_t) \;-\; F(X^N_s) \;-\; 
\int_s^t L^{\mc E}_N
(F\circ\Psi_N)(\eta^{\mc E_N}_r) \, dr \Big\} \Big ] \;=\; 0\;.
\end{equation*}
Since $\bb Q^N_{\,\eta_N}$ converges to $\bb Q\,$, time averages of
$E_{\bb Q^N_{\eta_N}}[F(X_t) U]$, $E_{\bb Q^N_{\eta_N}}[F(X_s) U]$
converge to time averages of $E_{\bb Q}[F(X_t) U]$, $E_{\bb Q}[F(X_s)
U]$, respectively.

We first claim that
\begin{equation}
\label{f08}
\lim_{N\to\infty}
\bb E^{\mc E}_{\eta_N} \Big [U \int_s^t \Big\{ (L^{\mc E}_N
(F\circ\Psi_N))(\eta^{\mc E_N}_r)  -
(\mf L_N F)(X^N_r)) \Big\} \, dr \Big]\;=\;0\;.
\end{equation}
Since $U$ is bounded, by the Markov property, it is enough to show
that
\begin{equation*}
\lim_{N\to\infty} \max_{\eta \in \mc E_N} \Big\vert \,\bb E^{\mc E}_{\eta} 
\Big [\int_0^{t-s} \Big\{ (L^{\mc E}_N
(F\circ\Psi_N))(\eta^{\mc E_N}_r) - (\mf L_N F)(X^N_r)) \Big\} \, dr \Big] \,
\Big\vert \;=\; 0\;.
\end{equation*}
By definition of the operator $\mf L_N$ and since 
\begin{equation*}
L^{\mc E}_N (F\circ\Psi_N)(\eta) \;=\; 
\sum_{y\in S\setminus\{x\}} \big\{F(y) - F\circ\Psi_N(\eta) \big\} 
\, R _N^{\mc E}(\eta, y)\;,\quad \textrm{for $\eta\in \mc E^x_N$,}
\end{equation*}
the difference $L^{\mc E}_N (F\circ\Psi_N)(\eta) - (\mf L_N
F)(\Psi_N(\eta))$ is equal to
\begin{equation*}
\sum_{y\in S} \big\{F(y) - F(\Psi_N(\eta))\big\} \, \{ R_N^{\mc E}(\eta,y) - 
r_N(\Psi_N(\eta),y)\}\;. 
\end{equation*}
In particular, \eqref{f08} follows from Lemma \ref{s03} and condition
{\bf (C2)} because
$$
\bb E^{\mc E}_{\eta}[H_{\xi^x_N}] \;\le\; \bb E^N_{\eta}[H_{\xi^x_N}]
$$
for every $\eta\in \mc E^x_N$. On the other hand, from condition {\bf
  (C1)}, it follows that
\begin{equation*}
\lim_{N\to\infty}
\bb E^{\mc E}_{\eta_N} \Big [U \int_s^t \Big\{ (\mf L_N F)(X^N_r) -
(\mf L F)(X^N_r) \Big\} \, dr \Big]\;=\;0\;.
\end{equation*}
$\bb Q^N_{\,\eta_N}$ converges to $\bb Q\,$ by assumption. Hence,
putting together the previous estimates we obtain that
\begin{equation*}
\lim_{N\to\infty}
\bb E^{\mc E}_{\eta_N} \Big [U \int_s^t (L^{\mc E}_N
(F\circ\Psi_N))(\eta^{\mc E_N}_r) \, dr \Big]
\;=\; \bb E_{\bb Q\,} \Big [U \int_s^t (\mf L F)(X_r) \, dr \Big]
\end{equation*}
because $U$ is bounded and continuous for the Skorohod topology.

Up to this point we proved that 
\begin{equation*}
\frac 1\epsilon \int_0^\epsilon dr \, 
E_{\bb Q} \Big [ U \Big\{ F(X_{t+r}) \;-\; F(X_{s+r}) \;-\; 
\int_{s+r}^{t+r} (\mf L F)(X_s) \Big\} \Big ] \;=\; 0
\end{equation*}
for every $\epsilon >0$. It remains to let $\epsilon \downarrow 0$ and
use the right continuity of the process to deduce \eqref{f11}, which
concludes the proof of the lemma.
\end{proof}

\noindent\textbf{Condition (B).}
We need to introduce some additional notation. For a path $\omega\in
D(\bb R_+,S)$ performing infinitely many jumps, we denote by
$\tau_n(\omega)$, $n\ge 0$ the jumping times of $\omega$:
$\tau_0(\omega)=0$ and
\begin{equation}\label{taun}
\tau_n(\omega) \;:=\; \inf\{t>\tau_{n-1}(\omega) : 
\omega(t)\neq \omega(\tau_{n-1}(\omega)) \}\;.
\end{equation}
Let us denote 
$$
T_n(\omega) \; := \; \tau_n(\omega)-\tau_{n-1}(\omega)\;,\quad n\ge 1\;.
$$
Further, denote by $N_t$ the number of jumps up to time $t$:
$$
N_t(\omega) \; := \; \sup\{j\ge 0 : \tau_j(\omega)\le t\}\;.
$$
The process $X_t^N:=\Psi_N(\eta^{\mc E_N}_t)$ is defined on the path
space $D(\bb R_+,\mc E_N)$. Nevertheless, $X^N_t$ can be think of as a
process defined on $D(\bb R_+,E_N)$ in the obvious way. Hence, we
obtain a coupling of $X^N_t$ and $\hat X^N_t$, where $\hat X^N_t$ has
been defined defined just before Proposition \ref{corol}. In what
follows we consider this coupling.

\begin{lemma}
Condition {\rm \bf (C3)} implies that for every $n\ge 1$
\begin{equation}\label{Tn}
\lim_{N\to\infty}\max_{\eta\in\mc E_N}\bb E^N_{\,\eta}
[\, T_n(\hat X^N) - T_n(X^N)\,] \;=\; 0\;.
\end{equation}
\end{lemma}

\begin{proof}
By the strong Markov property, it suffices to prove the limit
(\ref{Tn}) for $n=1$. Let $\omega$ be a path in $D(\bb R_+,E_N)$. Fix
$x\in S$ and $\eta\in \mc E^x_N$. Under $\bb P_{\eta}^N$, if
$H_{\check{\mc E}^x_N}(\omega)< H_{\Delta_N}(\omega)$ then $T_1(\hat
X^N(\omega)) = T_1( X^N(\omega))$ and otherwise
\begin{equation}\label{t1}
T_1(\hat X^N(\omega)) - T_1(X^N(\omega)) \;=\; 
H_{\check{\mc E}^x_N}\circ\mc R^{E_N\setminus\mc E^x_N}\circ\Theta(H_{\Delta_N})(\omega)
\end{equation}
where  $\Theta(H_{\Delta_N})(\omega)\in D(\bb R_+,E_N)$ is defined by 
$$
\Theta(H_{\Delta_N})(\omega)(t) \;=\; \omega(H+t)
$$
and $H:=H_{\Delta_N}(\omega)$. By using (\ref{t1}) and the strong
Markov property,
\begin{eqnarray*}
\bb E^N_{\,\eta}[\, T_1(\hat X^N) - T_1(X^N)\,] 
&\le& \max_{\xi\in\partial_x\Delta_N}\bb E_{\xi}^N[H_{\check{\mc
    E}^x_N} \circ\mc R^{E_N\setminus\mc E^x_N}] \\
&=& \max_{\xi\in\partial_x\Delta_N}\bb E_{\xi}^{N,x}[H_{\check{\mc E}^x_N}]
\end{eqnarray*}
Hence, (\ref{Tn}) for $n=1$ follows from condition {\bf (C3)}.
\end{proof}

Now, we proceed to check condition (B). To keep notation simple, set
$\hat N_t:=N_t(\hat X^N)$, $\hat T_n:=T_n(\hat X^N)$ and
$T_n:=T_n(X^N)$. Fix an arbitrary $\eta\in\mc E_N$. Under $\bb
P^N_{\eta}$, it is easy to see that
\begin{eqnarray*}
\mc T^{\Delta_N}_t &\le& t \land \sum_{n=1}^{\hat N_t +1}
(\hat T_n - T_n) \\
&\le& \mathbf 1\{\hat N_t \ge K\} \; t \;+\; 
\mathbf 1\{\hat N_t < K\} \sum_{n=1}^{K}(\hat T_n - T_n)\;,
\end{eqnarray*}
for any positive integer $K$. Therefore,
\begin{equation}\label{e01}
\bb E^N_{\eta}[\mc T^{\Delta_N}_t] \; \le \; t \;
\bb P^N_{\eta}[\,N_t(\hat X^N) \ge K \,] \;+\; 
\bb E^N_{\eta}[\,\tau_K(\hat X^N) - \tau_K(X^N)\,]\;.
\end{equation}
It follows from the previous lemma that
\begin{equation*}
\lim_{N\to \infty}\max_{\eta\in \mc E_N}\bb E^N_{\eta}[\,\tau_K(\hat
X^N) - \tau_K(X^N)\,] \;=\; 0\;.
\end{equation*}
Hence, in view of estimate (\ref{e01}), it remains to prove that
\begin{equation}\label{e02}
\lim_{K\to\infty} \limsup_{N\to\infty} \max_{\eta\in \mc E_N} 
\bb P^N_{\eta}[\,N_t(\hat X^N) \ge K \,] \;=\;0\;.
\end{equation}
For it, notice that $N_t(\hat X^N)\le N_t(X^N)$, $\bb P^N_{\eta}$ -
a.s. for all $\eta\in\mc E_N$ and, in consequence,
$$
\bb P^N_{\eta}[\,N_t(\hat X^N) \ge K \,] \;\le\; 
\bb P^N_{\eta}[\,N_t(X^N) \ge K \,]\;.
$$
Fix the sequence $\eta_N\in \mc E_N$, $N\ge 1$ defined by
\begin{eqnarray}
\max_{\eta\in \mc E_N} \bb P^N_{\eta}[\,N_t(X^N)\ge K \,] 
&:=&  \bb P^N_{\eta_N}[\,N_t(X^N)\ge K \,] \nonumber\\
&=& \bb Q^N_{\,\eta_N}[N_t\ge K].\label{e03}
\end{eqnarray}
Now, it is not difficult to see that $[N_t\ge K]$ is a closed set for
the Skorohod topology on $D(\bb R_+,S)$. Then,
\begin{equation}\label{e04}
\limsup_{N\to \infty}\bb Q^N_{\,\eta_N}[N_t\ge K] \;\le\; 
\max_{x\in S} \bb Q_{\,x}[N_t\ge K]\;.
\end{equation}
Finally, since $S$ is a finite set, limit (\ref{e02}) follows from
(\ref{e03}) and (\ref{e04}).  This concludes the proof of Theorem
\ref{teo1}.

\subsection{Proof of Theorem \ref{teo2}}

Recall that $\nu_N$ is supposed to be a reversible probability measure
for all $N\ge 1$. We recall the formula for the expected value of
entry times proved in \cite{b2} : For any subset $F$ of $E_N$ and any
$\eta\in E_N \setminus F$,
\begin{equation}\label{anton}
\bb E^N_{\eta}[H_F] \; = \; \frac{\nu_N(f_{\eta,F})}{\Cap_N(\eta,F)}\;\cdot
\end{equation}
This is the key result in order to introduce the notion of capacity in
our study of meta-stability.

In what follows, we check each condition in Theorem \ref{teo1}.
Obviously, condition {\bf (C1)} is already in force. To check
condition {\bf (C2)}, fix two elements $x,y$ of $S$, $x\not = y$. By
using the strong Markov property, we have for any two sites
$\eta,\xi\in \overline{\mc E^x_N}$,
$$
\bb E^N_{\eta}[H_{\xi}] \;\le\; \bb E^N_{\eta}[H_{\xi^x_N}] 
+ \bb E^N_{\xi^x_N}[H_{\xi}]\;.
$$
Then, by using formula (\ref{anton}),
\begin{equation}\label{equa0}
\max_{\eta,\xi\in \overline{\mc E^x_N}}\bb E^N_{\eta}[H_{\xi}] 
\;\le\; \frac{2}{\min\{\Cap_N(\eta,\xi^x_N) : 
\eta\in \overline{\mc E^x_N}\}} \;\cdot
\end{equation}
On the other hand, if $\partial\mc E_N^{x,y}$ is empty then
$\max_{\eta\in\mc E^x_N}R_N^{\mc E}(\eta,y)=0$. Otherwise, an
elementary estimate shows that
\begin{equation}\label{equa1}
\max_{\eta\in\mc E^x_N}R_N^{\mc E}(\eta,y) \;\le\; 
\frac{\nu_N(\mc E^x_N) r_N(x,y) }{\min\{\nu_N(\eta) : 
\eta\in\partial\mc E^{x,y}_N\}}\;\cdot
\end{equation}
Therefore, condition {\bf (C2)} follows from (\ref{equa0}),
(\ref{equa1}) and hypotheses {\bf (H3)} because $r_N(x,y)$ is bounded
by virtue of hypotheses {\bf (C1)}.

Finally we check condition {\bf (C3)}. Fix $x\in S$. By using the
strong Markov property for the trace of $\eta^N_t$ on $E_N\setminus
\mc E^x_N$,
$$
\max_{\eta\in\partial_x\Delta_{N}}
\bb E^{N,x}_{\,\eta}[H_{\check{\mc E}^x_N}] 
\;\le\; \max_{\eta\in\partial_x\Delta_{N}}
\bb E^{N,x}_{\,\eta}[H_{\zeta^x_N}] \;+\; 
\bb E^{N,x}_{\,\zeta^x_N}[H_{\check{\mc E}^x_N}]\;.
$$
By hypotheses {\bf (H3)} and estimate (\ref{equa0}), the first term
in the right hand side of this inequality vanishes as
$N\uparrow\infty$, because $\bb E^{N,x}_{\,\eta}[H_{\zeta^x_N}]\le\bb
E^{N}_{\,\eta}[H_{\zeta^x_N}]$ for all $\eta\in\partial_x\Delta_N$.
Moreover, by using the formula in (e) of Lemma \ref{s07},
\begin{eqnarray*}
\bb E^{N,x}_{\,\zeta^x_N}[H_{\check{\mc E}^x_N}] 
&=& \frac{\nu_N\big(\;\mb 1\{E_N\setminus \mc E^x_N\}\;
f_{\zeta^x_N,\check{\mc E}^x_N}\;\big)}
{\Cap_{N}(\zeta^x_N,\check{\mc E}^x_N)} \\
&\le& \frac{\nu_N(\Delta_N)}{\Cap_{N}(\zeta^x_N,\check{\mc E}^x_N)}\;\cdot
\end{eqnarray*}
In view of {\bf (H2)}, this concludes the proof.

\subsection{Proof of Proposition \ref{corol}}
At first, we fix a metric in $D(\bb R_+,S)$ which induces the Skorohod
topology. For each integer $m\ge 1$, let $\Lambda_m$ denote the class
of strictly increasing, continuous mappings of $[0,m]$ onto itself. If
$\lambda \in \Lambda$, then $\lambda_0=0$ and $\lambda_m=m$. For
$\omega$ and $\hat\omega$ in $D(\bb R_+,S)$, define
$d_m(\omega,\hat\omega)$ to be the infimum of those positive
$\epsilon$ for which there exists in $\Lambda_m$ a $\lambda$
satisfying
$$
\sup_{t\in [0,m]} |\lambda_t-t| < \epsilon
$$
and
$$
\sup_{t\in [0,m]}|\omega(\lambda_t) - \hat\omega(t)| < \epsilon\;.
$$
Now, for each integer $m\ge 1$, define
$$
g_m(t)\;=\;\left\{ 
\begin{array}{ll} 
1 & \textrm{if $t\le m-1$}\;,\\
m-t & \textrm{if $m-1 \le t\le m$\;,}\\
0 & \textrm{if $t\ge m$}\;.
\end{array}\right.
$$
For each $\omega\in D(\bb R_+,S)$ let $\omega^m$ the element of
$D(\bb R_+,S)$ defined by
$$
\omega^m(t)\;=\;g_m(t)\omega(t),\quad t\ge 0\;.
$$
Finally, we define the metric in $D(\bb R_+,S)$ by
$$
d(\omega,\hat\omega) \;=\; \sum_{m=1}^{\infty} 
2^{-m}(1\land d_m(\omega^m,\hat\omega^m))\;.
$$
This metric induces the Skorohod topology in $D(\bb R_+,S)$ (cf. \cite{b}).

Recall the jumping times $\tau_n(\omega)$, $n\ge 0$ defined in
(\ref{taun}) for any path $\omega\in D(\bb R_+,S)$ performing
infinitely many jumps. Further, let $n_t(\omega)$ stand for the number
of jumps of $\omega$ strictly before time $t>0$ :
$$
n_t(\omega)\;:=\;\sup\{j\ge 0 : \tau_j(\omega) < t\}\;.
$$

Now, let $\omega$ and $\hat\omega$ be elements of $D(\bb R_+,S)$ which
perform infinitely many jumps. To keep notation simple set
$\tau_n:=\tau_n(\omega)$, $\hat\tau_n:=\tau_n(\hat\omega)$ and $\hat
n_t:=n_t(\hat\omega)$. Suppose that
\begin{equation}\label{cono}
\hat\tau_{n+1}-\hat \tau_{n} \;\ge\; \tau_{n+1} 
- \tau_n\quad \textrm{and}\quad \omega(\tau_n) 
\;=\; \hat\omega(\hat\tau_n)\quad \textrm{for all $n\ge 0\;,$}
\end{equation}
so that $\omega(s)=\hat\omega(t)$ for $(s,t)\in
[\tau_n,\tau_{n+1}[\times[\hat\tau_n,\hat\tau_{n+1}[$ and $n\ge 0.$
Then, we claim that
\begin{equation}\label{dm}
d_m(\omega^m,\hat\omega^m) \;\le\; \kappa 
\max\{\hat\tau_{\hat n_m}-\tau_{\hat n_m},(m-\tau_{\hat n_m+1})^{+} \}
\end{equation}
for all integer $m\ge 1$. Let us prove the claim. Notice that
$\tau_{\hat n_m}<m$ for every $m\ge 1$. Choose $\epsilon>0$ small
enough such that $\hat\tau_{\hat n_m}<m-\epsilon$ and $\tau_{\hat n_m}
< (m\land \tau_{\hat n_m+1})-\epsilon$. Now, let $\lambda\in
\Lambda_m$ be given by: $\lambda_{\hat\tau_j}=\tau_j$, for $j\le \hat
n_m$, $\lambda_{m-\epsilon}=(m\land \tau_{\hat n_m+1})-\epsilon$,
$\lambda_m=m$ and we complete $\lambda$ on $[0,m]$ by linear
interpolation. Then,
$$
\sup_{t\in [0,m]} |\lambda_t-t| \; \le \; 
\max\{ \hat\tau_{\hat n_m} - \tau_{\hat n_m}, (m-\tau_{\hat n_m+1})^+\}\;.
$$
Moreover,
\begin{eqnarray*}
\sup_{t\in[0,m-\epsilon]}\big|g(\lambda_t)\omega(\lambda_t)
-g(t)\hat\omega(t)\big| &\le& \kappa\sup_{t\in[0,m-\epsilon]} 
|g(\lambda_t)-g(t)|\\
&\le& \kappa \max\{ \hat\tau_{\hat n_m} - \tau_{\hat n_m}, 
(m-\tau_{\hat n_m+1})^+\}
\end{eqnarray*}
and
\begin{eqnarray*}
\sup_{t\in[m-\epsilon,m]}\big|g(\lambda_t)
\omega(\lambda_t)-g(t)\hat\omega(t)\big| &\le& 
\kappa \sup_{t\in[m-\epsilon,m]}\big(\, |g(\lambda_t)| + |g(t)|\,\big)\\
&\le& \kappa \{ (m-\tau_{\hat n_m+1})^+ + 2\epsilon \}\;.
\end{eqnarray*}
Since $\epsilon$ is arbitrary small, the claim is proved.

Recall the coupling of $\hat X^N_t$ and $X^N_t$ used in the proof of
Condition (B). Then, in order to prove Proposition \ref{corol}, it is
enough to prove that for any $\delta>0$
\begin{equation*}
\lim_{N\to\infty}\bb P^{N}_{\eta_N}\big[\,d(X^N,\hat X^N)
>\delta\,\big] \;=\; 0\;.
\end{equation*}
Notice that the paths of $\hat X^N_t$ and $X^N_t$ in the coupling
satisfy conditions (\ref{cono}) and, in consequence, estimate
(\ref{dm}). Let us denote $ \hat n^N_m := n_m(\hat X^N)$. To estimate
the right hand side in (\ref{dm}), observe that
$$
\tau_{\hat n^N_m}(\hat X^N) - \tau_{\hat n^N_m}(X^N) 
\;=\; \mc T_{\hat n^N_m}^{\Delta_N} \le \mc T^{\Delta_N}_m
$$
and $(m-\tau_{\hat n^N_m+1}(X^N))^+ \le \mc T_m^{\Delta_N}$.
Therefore, by (\ref{dm}) we have
$$
d(X^N,\hat X^N)\; \le \; \sum_{m=1}^{\infty} 2^{-m} 
(1\land \kappa\mc T^{\Delta_N}_m)\;.
$$
 From this estimate and condition (\ref{metab}) of 
meta-stability it follows that
$$
\lim_{N\to\infty} \bb E^N_{\eta_N}[d(X^N,\hat X^N)] \;=\; 0\;.
$$
This completes the proof of Proposition \ref{corol}.


\section{Markov processes}
\label{sec03}
We state in this section several properties of Markov processes used
throughout the article.  Consider an irreducible Markov process $\{
\eta_t : t\ge 0 \}$ on a finite state space $E$ and denote its unique
invariant measure by $\nu$.  Denote by $\bb P_{\eta}$ the probability
measure on the path space $D(\bb R_+,E)$ corresponding to the Markov
process $\eta_t$ starting from $\eta$ and let $\bb E_{\eta}$ stand for
the expectation with respect to $\bb P_{\eta}$.

\subsection{Replacement Lemma}
In this subsection we prove the result used in Lemma \ref{s03}. We
start with the following identity.

\begin{lemma}
\label{s08}
Fix a mean-zero function $V: E \to \bb R$ and a site $\eta$ in $E$.
Let $H_{\eta}$ be the entry time to $\eta$ : $H_{\eta} \;=\; \inf
\{t\ge 0 : \eta_t = \eta\}.$ Then, for all $\xi$ in $E$,
\begin{equation*}
\bb E_{\xi} \Big[ \int_0^t  V(\eta_s) \, ds \Big] \;=\; 
\bb E_{\xi} \Big[ \int_0^{H_{\eta}}  V(\eta_s) \, ds \Big] \;-\;
\bb E_{\xi} \Big[ \bb E_{\eta_t} \Big[ \int_0^{H_{\eta}}  V(\eta_s) \, ds \Big] 
\, \Big]\;.
\end{equation*}
\end{lemma}

\begin{proof}
Let $(\Theta_t)_{t\ge 0}$ stand for the shift operators. Define $H_0=0$,
$$
H_{1} \;=\; \inf \{t\ge 0 : \eta_t = \eta \,, \eta_s \not = \eta \text{
  for some $0\le s < t$}\}\;
$$
and $H_{j+1}=H_{1}\circ \Theta_{H_{j}}$ for $j\ge 1$.
The expected value appearing in the
statement of the lemma can be written as
\begin{eqnarray}
\label{f17}
\!\!\!\!\!\!\!\!\!\!\!\!\!\!\! &&
\sum_{j\ge 0} \bb E_{\xi} \Big[ \int_0^t  V(\eta_s) \, ds 
\, \mb 1\{H_j \le t <H_{j+1} \} \Big] \\
\!\!\!\!\!\!\!\!\!\!\!\!\!\!\! && \quad =\;
\sum_{j\ge 0} \sum_{k=0}^j \bb E_{\xi} \Big[
\int_{H_k}^{H_{k+1}} V(\eta_s) \, ds \, \mb 1\{H_j \le t <H_{j+1}
\} \Big] \nonumber \\
\!\!\!\!\!\!\!\!\!\!\!\!\!\!\! && \quad - \;
\sum_{j\ge 0} \bb E_{\xi} \Big[
\int_{t}^{H_{j+1}} V(\eta_s) \, ds \, \mb 1\{H_j \le t <H_{j+1}
\} \Big] \;. \nonumber
\end{eqnarray}

We compute these two terms separately. The first one can be rewritten
as
\begin{equation*}
\sum_{k\ge 0} \bb E_{\xi} \Big[ \int_{H_k}^{H_{k+1}} V(\eta_s) 
\, ds \, \mb 1\{H_k \le t \} \Big]\;.
\end{equation*}
Taking conditional expectation with respect to $\mc F_{H_k}$, by the
strong Markov property, this sum is equal to
\begin{equation*}
\bb E_{\xi} \Big[ \int_{0}^{H_1} V(\eta_s) \, ds \Big]
\;+\; \bb E_{\eta} \Big[ \int_{0}^{H_1} V(\eta_s) \, ds \Big] \, 
\sum_{k\ge 1} \bb P_{\xi} \Big[ H_k \le t \Big]  \;.
\end{equation*}
Notice that in the second expectation, $\eta$ appears instead of
$\xi$.  Since $V$ has mean-zero with respect to the unique stationary
measure $\nu$, the second expectation vanishes.

We now turn to the second term in the decomposition \eqref{f17}. Observe
that $\{H_j \le t < H_{j+1}\}$ belongs to the $\sigma$-algebra $\mc F_t$
and on this set $H_{j+1} = t + H_1 \circ \Theta_t$. Therefore, by the
Markov property, the second term in \eqref{f17} is equal to
\begin{equation*}
\sum_{j\ge 0} \bb E_{\xi} \Big[ \mb 1\{H_j \le t <H_{j+1}
\} \, \bb E_{\eta_t} \Big[ \int_{0}^{H_1} V(\eta_s) \, ds \,
\Big]\, \Big]  
\;=\; \bb E_{\xi} \Big[ \bb E_{\eta_t} 
\Big[ \int_{0}^{H_1} V(\eta_s) \, ds \, \Big] \, \Big] \;.
\end{equation*}
We have just proved that
\begin{equation*}
\bb E_{\xi} \Big[ \int_0^t  V(\eta_s) \, ds \Big] \;=\; 
\bb E_{\xi} \Big[ \int_0^{H_1}  V(\eta_s) \, ds \Big] \;-\;
\bb E_{\xi} \Big[ \bb E_{\eta_t} \Big[ \int_0^{H_1}  V(\eta_s) \, ds \Big] 
\, \Big]\;.
\end{equation*}
On the one hand, if $\xi\not = \eta$ then $H_{\eta}=H_1$ $\bb
P_{\xi}$-a.s. and so
$$
E_{\xi} \Big[ \int_0^{H_1} V(\eta_s) \, ds \Big] 
\;=\; E_{\xi} \Big[ \int_0^{H_{\eta}}  V(\eta_s) \, ds \Big]\;.
$$
On the other hand, if $\xi = \eta$ then $H_{\eta}=0$ $\bb
P_{\xi}$-a.s. and so
$$
E_{\eta} \Big[ \int_0^{H_1} V(\eta_s) \, ds \Big] 
\;=\; 0 \;=\; E_{\eta} \Big[ \int_0^{H_{\eta}}  V(\eta_s) \, ds \Big]\;.
$$
This concludes the proof of the lemma.
\end{proof}

\begin{corollary}
\label{s09}
Assume further that $V$ vanishes outside a subset $F$ of $E$. Then,
\begin{equation*}
\max_{\xi\in E} \Big\vert \bb E_{\xi} \Big[ \int_0^t  
V (\eta_s) \, ds \Big] \, \Big\vert \;\le\; 2 \Vert V\Vert_\infty
\max_{\xi\in F} \bb E_{\xi} [H_{\eta}] \;.
\end{equation*}
\end{corollary}

\begin{proof}
We need to estimate the two expectations appearing on the right hand
side of the statement of Lemma \ref{s08}. Recall that $H_{F}$ denotes
the entry time in $F$ defined in (\ref{f10}). Since $ V$ vanishes
outside $F$, in the first expectation we may take the time integral on
the interval $[H_{F},H_{\eta}]$ and apply the strong Markov property
to rewrite this term as
\begin{equation*}
\bb E_{\xi} \Big[ \bb E_{\eta_{H_{F}}} \Big[ 
\int_{0}^{H_{\eta}} V(\eta_s) \, ds \Big] \, \Big]\;.
\end{equation*}
This expression is absolutely bounded by $\Vert V\Vert_\infty
\max_{\xi\in F} \bb E_{\xi} [ H_{\eta}]$. By the same reasons, the second
term is bounded above by the same quantity.
\end{proof}

\subsection{The trace process}\label{trace}
We recall in this subsection some elementary properties of trace
processes and deduce from these attributes some identities used
throughout the article.

Denote by $L$ the generator of the Markov process $\eta_t$ and by
$R(\cdot, \cdot)$ its transition rates so that
\begin{equation*}
(Lf)(\eta) \;=\; \sum_{\xi\in E} \big\{ \,f(\xi) - f(\eta)\,\big\}\,R(\eta,\xi)
\end{equation*}
for all functions $f:E\to\bb R$. Let $p(\cdot,\cdot)$ be the jump
probabilities and $\lambda$ the holding times:
\begin{equation*}
\lambda(\eta) \;=\; \sum_{\xi\not = \eta} R(\eta,\xi)\;, \quad
p(\eta,\xi) \;=\; \frac 1{\lambda(\eta)} \, R(\eta,\xi)\;.
\end{equation*}

Fix a non-empty subset $F$ of $E$. Denote by $\eta^F_t$ the trace
process of $\eta_t$ on $F$. Denote by $L^F$ the generator of the
Markov process $\eta^F_t$, by $R^F(\cdot, \cdot)$ its transition
rates, by $\lambda^F$ the holding times and by $p^F(\cdot, \cdot)$ the
jump probabilities. We have that
\begin{equation}
\label{g02}
\lambda^{F}(\eta) \;=\; \, \lambda(\eta) \, 
\bb P_{\eta} \big[ H_{F\setminus\{\eta\}} < \tau_{\eta}\big] \; , 
\quad p^{F}(\eta,\xi) \;=\; 
\bb P_{\eta} \big[ H_{\{\xi\}} = H_{F\setminus\{\eta\}}\big]
\end{equation}
for $\eta,\xi\in F$, $\eta\not = \xi$, where $\tau_{\eta}$ stands for
the time of the first return to $\eta$ :
\begin{equation*}
\tau_{\eta} \;=\; \inf \big\{t\ge 0 : \eta_t = \eta \,, 
\eta_s \not = \eta \text{
  for some $0\le s < t$}\big\}\;.
\end{equation*}

To prove the first identity in \eqref{g02}, denote by $\tau^F_1$ the
first jumping time of the process $\eta^F_t$. Under $\bb P_{\eta}$,
$\tau^F_1$ can be represented as
\begin{equation*}
\tau^F_1 \;=\; \sum_{j=1}^N S_j\;,
\end{equation*}
where $S_j$ are i.i.d.\! exponential random variables with parameter
$\lambda(\eta)$ and, by the strong Markov property, $N$ is a geometric
r.v. independent of the sequence $\{S_j : j\ge 1\}$ and such that $\bb
P_\eta [N=k] = a (1-a)^{k-1}$, $k\ge 1$, for $a= \bb P_{\eta} \big[
H_{F\setminus\{\eta\}} < \tau_{\eta}\big]$. It remains to recall that
an independent geometric sum of independent exponential r.v. is
exponentially distributed. The second identity in \eqref{g02} is
obvious.
 
An explicit formula for $\lambda^{F}$, $p^{F}$ can be obtained
recursively. Indeed, an elementary computation shows that, when $F=
E\setminus \{\xi_0\}$ for some $\xi_0$ in $E$,
\begin{equation}\label{eqRF}
R^F(\eta,\xi) \;=\; R(\eta,\xi) \;+\; R(\eta,\xi_0) \, p(\xi_0,\xi)
\end{equation}
for $\eta\not = \xi$, $\{\eta,\xi\}\subseteq E\setminus \{\xi_0\}$.

 From now on, we assume that $\eta_t$ is an irreducible Markov
process. Denote by $\nu$ its unique invariant measure.

\begin{lemma}
\label{trnu}
For any non-empty subset $F$ of $E$, $\{\eta^F_t : t\ge 0\}$ is an
irreducible Markov process. The measure $\nu$ conditioned to $F$,
denoted by $\nu^F$, is the unique invariant probability measure of
$\eta^F_t$. Moreover, $\nu^F$ is reversible, whenever $\nu$ is.
\end{lemma}
\begin{proof}
For $F=E\setminus \{\xi_0\}$, the lemma follows from identity
(\ref{eqRF}) and straightforward computations. To prove it for general
subsets $F$ we proceed by induction on the cardinality of $F$.
\end{proof}

For each $F\subseteq E$, and all subsets $G_1$, $G_2$ of $F$ such that
$G_1 \cap G_2 = \phi$ we denote
\begin{equation*}
r_F(G_1,G_2) \;:=\; \frac{1}{\nu(G_1)} 
\sum_{\substack{\eta\in G_1 \\ \xi \in G_2}} \nu(\eta) \, R^F(\eta,\xi)\;.
\end{equation*}
Observe that if $\nu$ is reversible, then 
$$
\nu(G_1) \; r_F(G_1,G_2)\;=\;\nu(G_2) \; r_{F}(G_2,G_1)\;.
$$
Recall the definition of the capacity of irreducible reversible Markov
processes introduced in Subsection \ref{cap}. We denote by
$\Cap(\cdot,\cdot)$, respectively by $\Cap_F(\cdot,\cdot)$, the
capacity for $\eta_t$, respectively for $\eta^F_t$.

We summarize in the next lemma, properties of the process $\eta^F_t$
used throughout Section \ref{proof}. For each $\eta$ in $F$, let $\bb
P^F_{\eta}$ be the probability measure on the path space $D(\bb R_+,
F)$ induced by the Markov process $\eta^F_t$ starting from
$\eta$. Also, denote by $\bb E^F_{\eta}$ the expectation with respect
to $\bb P^F_{\eta}$ for each $\eta\in F$.

\renewcommand{\theenumi}{\alph{enumi}}
\renewcommand{\labelenumi}{(\theenumi)}

\begin{lemma}
\label{s07}
Fix a non-empty subset $F$ of $E$, and two subsets $G_1$, $G_2$ of $F$
such that $G_1 \cap G_2 = \phi$. There holds,
\begin{enumerate}
\item For all $\eta$ in $F$,
\begin{equation*}
\bb P^F_{\eta} \big [H_{G_1} < H_{G_2} \big] \;=\; 
\bb P_{\eta} \big [H_{G_1} < H_{G_2} \big] \;.
\end{equation*}
\item Let $\check{G}_{2} = F\setminus G_2$, then
\begin{equation*}
\nu(G_1) \; r_F(G_1,G_2) \;=\;
\sum_{\substack{\eta\in G_1 \\ \xi \in E}}
\nu(\eta) \, R(\eta,\xi) \, \bb P_{\xi}\big [ H_{G_2} <H_{\check{G}_{2}} \big]\;.
\end{equation*}
\item If $\nu$ is reversible and $F=G_1\cup G_2$, then
$$
\nu(G_1)\; r_F(G_1,G_2)\;=\; \Cap(G_1,G_2)\;.
$$
\item If $\nu$ is reversible, then
\begin{equation*}
\Cap_F(G_1,G_2)\;=\;\frac {\Cap(G_1,G_2)}{\nu(F)}\;\cdot
\end{equation*}
\item If $\nu$ is reversible, then for any $\eta\in F\setminus G$
\begin{equation*}
\bb E^F_{\eta}[H_{G}]\;=\;\frac{\sum_{\xi\in F}
\bb P_{\xi} \big [H_{\{\eta\}} < H_{G} \big]\;\nu(\xi) }
{\Cap\big(\{\eta\},G\big)}\;\cdot
\end{equation*}
\end{enumerate}
\end{lemma}

\begin{proof}
Claim (a) is obvious from the construction of the trace process. To
prove Claim (b), observe that we may write the left hand side of (b)
as
\begin{equation*}
\sum_{\substack{\eta\in G_1 \\ \xi \in F}} \nu(\eta) \, R^F(\eta,\xi) \,
\bb P^F_{\xi}\big [ H_{G_2} < H_{\check{G}_{2}} \big]\;.
\end{equation*}
Enumerate the set $E\setminus F$ as $\{\xi_1, \dots, \xi_M\}$. Let
$E_0 =\phi$, $E_i = \{\xi_1, \dots , \xi_i\}$, for $1\le i\le M$.  We
prove by induction that the previous expression is equal to
\begin{equation}
\label{f12}
\sum_{\substack{\eta\in G_1 \\ \xi \in F \cup E_i}} 
\nu(\eta) \, R^{F \cup E_i}(\eta,\xi) 
\, \bb P^{F\cup E_i}_{\xi}\big [ H_{G_2} < H_{\check{G}_{2}} \big]
\end{equation}
for $0\le i\le M$. Claim (b) follows from this identity since the
right hand side of (b) coincide with the previous sum for $i=M$.

Obviously, the two previous sums are equal for $i=0$. 
Assume that the identity holds for some $0\le i\le M-1$. 
By (\ref{eqRF}), \eqref{f12} is equal to
\begin{eqnarray*}
\!\!\!\!\!\!\!\!\!\!\!\!\!\! &&
\sum_{\substack{\eta\in G_1 \\ \xi \in F \cup E_i}} \nu(\eta) \, R^{F \cup
  E_{i+1}}(\eta,\xi)  
\bb P^{F\cup E_i}_{\xi}\big [ H_{G_2} < H_{\check{G}_{2}} \big] \\
\!\!\!\!\!\!\!\!\!\!\!\!\!\! && \quad
\;+\; \sum_{\substack{\eta\in G_1 \\ \xi \in F \cup E_i}} \nu(\eta) \, 
R^{F \cup E_{i+1}}(\eta,\xi_{i+1}) p^{F \cup E_{i+1}} (\xi_{i+1}, \xi) 
\bb P^{F\cup E_i}_{\xi}\big [ H_{G_2} < H_{\check{G}_{2}} \big] \;.
\end{eqnarray*}
By (a), we may replace in the previous formula $F\cup E_i$ by $F\cup E_{i+1}$ 
in the superscript of $\bb P_{\xi}$. After this replacement, note that
\begin{equation*}
\sum_{\xi \in F \cup E_i} p^{F \cup E_{i+1}} (\xi_{i+1}, \xi) 
\bb P^{F\cup E_{i+1}}_{\xi}\big [ H_{G_2} < H_{\check{G}_{2}} \big] 
\;=\; \bb P^{F\cup E_{i+1}}_{\xi_{i+1}} 
\big [ H_{G_2} < H_{\check{G}_{2}} \big] 
\end{equation*}
since the left hand side is obtained from the right hand side by
conditioning on the first jump. This proves (b).

To prove claims (c), (d) and (e), for the remaining part of the proof
we suppose that $\nu$ is reversible. By arguments similar to the proof
of Theorem II.6.1 in \cite{l}, we have
\begin{equation}\label{ki}
\Cap(G_1,G_2)\;=\;\sum_{\substack{\eta\in G_1 \\ \xi \in E}} 
\nu(\eta) \, R(\eta,\xi) \,
\bb P_{\xi}\big [ H_{G_2} < H_{G_1} \big]\;.
\end{equation}
Therefore, claim (c) follows from this identity and claim (b) when
$F=G_1\cup G_2.$ Further, it follows from the proof of (b) that
\begin{eqnarray*}
\sum_{\substack{\eta\in G_1 \\ \xi \in F}} \nu(\eta) \, R^F(\eta,\xi) \,
\bb P^F_{\xi}\big[ H_{G_2} < H_{G_1} \big] &=& 
\sum_{\substack{\eta\in G_1 \\ \xi \in E}} \nu(\eta) \, R(\eta,\xi) \,
\bb P_{\xi}\big [ H_{G_2} < H_{G_1} \big]\;.
\end{eqnarray*}
By the version of identity \eqref{ki} for the trace process
$\eta^F_t$, the left hand side of this equation is equal to
$\nu(F)\Cap_F(G_1,G_2)$. By \eqref{ki}, this completes the proof of
(d). Finally, claim (e) follows easily from formula (\ref{anton}) for
the trace process $\eta^F_t$, claim (a) and claim (d). We are done.
\end{proof}

\end{document}